\documentclass[a4paper,10pt]{amsart}
\usepackage[francais,english]{babel}
\usepackage{amsmath,amssymb,amsthm}
\usepackage[T1]{fontenc}
\usepackage{amsfonts}
\usepackage{vmargin}
\usepackage{cite}
\usepackage{pifont}
\usepackage{euscript}
\RequirePackage{mathrsfs}
\def\Q{\mathbf{Q}}\def\R{\mathbf{R}}\def\Z{\mathbf{Z}}
\def\qb{\overline{\Q}}

\DeclareMathOperator{\Hom}{Hom}
\DeclareMathOperator{\Ker}{Ker}
\DeclareMathOperator{\GL}{GL}
\DeclareMathOperator{\spec}{Spec}
\DeclareMathOperator{\card}{card}
\DeclareMathOperator{\vol}{vol}
\DeclareMathOperator{\vect}{vect}
\DeclareMathOperator{\ppcm}{ppcm}

\theoremstyle{plain}
\newtheorem{theo}{Th\'eor\`eme}[section]
\newtheorem{lemma}[theo]{Lemme}
\newtheorem{prop}[theo]{Proposition}
\newtheorem{coro}[theo]{Corollaire}

\theoremstyle{definition}
\newtheorem{defi}[theo]{D\'efinition}
\newtheorem{rema}[theo]{Remarque}

\title{Minima, pentes et alg\`ebre tensorielle}
\author{\'Eric Gaudron \& Ga\"el R\'emond}
\date{\today}
\begin{document}

\selectlanguage{francais}

\footnotetext{MSC~$2010$: 11G50 (14G40, 11E12, 05A10).}
\footnotetext{\textbf{Mots-clefs} : Fibr\'e ad\'elique hermitien,
pente d'Arakelov, pente maximale, minimum absolu,
produit tensoriel hermitien, puissance sym\'etrique,
puissance ext{\'e}rieure, lemme de Siegel absolu de Zhang,
th\'eor\`eme de Minkowski-Hlawka absolu, ppcm de multinomiaux.}
\begin{abstract}
Nous {\'e}tudions les relations entre la pente maximale d'un fibr{\'e} 
ad{\'e}lique hermitien, issue de la th{\'e}orie des pentes de Bost, et le 
minimum absolu de ce fibr\'e. En particulier,
nous {\'e}tablissons un th{\'e}or{\`e}me de Minkowski-Hlawka absolu
et nous montrons que le minimum absolu n'est pas multiplicatif par
produit tensoriel.
De plus nous montrons comment un lemme de Siegel d\^u \`a Zhang
permet de majorer la pente maximale d'un produit tensoriel de fibr\'es
ad\'eliques hermitiens et d'am\'eliorer ainsi un th\'eor\`eme
obtenu r\'ecemment par Chen. Nous en d{\'e}duisons des cons{\'e}quences
pour les puissances sym{\'e}triques et ext{\'e}rieures de tels fibr{\'e}s,
dont l'une fait intervenir le ppcm des coefficients multinomiaux, que nous
calculons.
\\[.2cm] \textsc{Abstract}.
Slopes of an adelic vector bundle exhibit a behaviour akin to
successive minima. Comparisons between the two amount to a Siegel
lemma. Here we use Zhang's version for absolute minima over the
algebraic numbers. We prove a Minkowski-Hlawka theorem in this
context. We also study the tensor product of two hermitian bundles 
bounding both its absolute minimum and maximal slope, thus
improving an estimate of Chen. We further include similar inequalities
for exterior and symmetric powers, in terms of some lcm of multinomial
coefficients.
\end{abstract}
\maketitle
\textbf{Coordonn\'ees des auteurs~:}
\vspace{0,3cm}

\noindent \'Eric \textsc{Gaudron}\\ \ding{41} Universit\'e
Grenoble I, Institut Fourier.\\ UMR $5582$, BP $74$\\ $38402$
Saint-Martin-d'H\`eres Cedex, France.\\
Courriel~: \texttt{Eric.Gaudron@ujf-grenoble.fr}\\
\ding{37} (33) 04 76 51 45 72\\
Page internet~: \texttt{http$:$//www-fourier.ujf-grenoble.fr/$\sim$gaudron}
\vspace{0,3cm}

\noindent Ga{\"e}l \textsc{R\'emond}\\ \ding{41} Universit\'e
Grenoble I, Institut Fourier.\\ UMR $5582$, BP $74$\\ $38402$
Saint-Martin-d'H\`eres Cedex, France.\\
Courriel~: \texttt{Gael.Remond@ujf-grenoble.fr}\\
\ding{37} (33) 04 76 51 49 86
\newpage
\section{Introduction}

Cet article \'etudie deux quantit\'es associ\'ees \`a un
fibr\'e ad\'elique hermitien : son
minimum absolu et sa pente maximale. Nous donnons plusieurs
in\'egalit\'es faisant intervenir ces deux nombres, en particulier
pour estimer leur comportement relativement au produit tensoriel des
fibr\'es et, subs\'equemment, aux puissances sym\'etriques
et ext\'erieures.

La notion de fibr\'e ad\'elique hermitien g\'en\'eralise \`a un corps
de nombres celle de 
r\'eseau euclidien avec laquelle elle co{\"\i}ncide sur $\Q$. Alors
qu'un r\'eseau euclidien est constitu\'e d'un $\Z$-module libre de
rang fini $\Omega$ et d'une norme euclidienne sur l'espace vectoriel
r\'eel $\Omega\otimes\R$, nous voyons plut\^ot un fibr\'e ad\'elique
hermitien sur un corps de nombres $k$ 
comme un couple form\'e d'un espace vectoriel $E$ sur $k$ et d'une
collection de normes $\|\cdot\|_v$ sur $E$ pour toute place $v$ de $k$
(la d\'efinition pr\'ecise est rappel\'ee dans la partie suivante) :
dans le cas d'un r\'eseau euclidien $E=\Omega\otimes\Q$, la norme
$\|\cdot\|_\infty$ est celle donn\'ee sur $\Omega\otimes\R$ tandis
que les normes aux places finies d\'efinissent $\Omega$ dans $E$.

On associe classiquement \`a un r\'eseau euclidien son (premier)
minimum : la plus petite norme d'un \'el\'ement non nul de
$\Omega$. Traduite dans le langage des fibr\'es ad\'eliques
hermitiens, cette d\'efinition 
devient la plus petite hauteur d'un \'el\'ement non nul de $E$ (la
hauteur s'obtient en faisant le produit des normes, voir
\S~\ref{defh}) mais nous souhaitons consid\'erer ici le minimum {\em
absolu} de $E$ c'est-\`a-dire l'infimum des hauteurs des \'el\'ements
non nuls de $E\otimes\qb$ : nous le notons $\Lambda(\overline{E},\qb)$.

Nous utilisons ensuite la pente d'un fibr\'e ad\'elique
hermitien. Dans le cas d'un r\'eseau 
euclidien, celle-ci est un avatar du covolume : si
$\overline{E}=(\Omega\otimes\Q,\|\cdot\|)$ est d\'eduit de $\Omega,\|\cdot\|$
alors $$\widehat{\mu}(\overline{E})=-\frac{1}{\dim
E}\log\vol(\Omega\otimes\R/\Omega)$$ o\`u le volume d\'esigne la
mesure de Haar sur $\Omega\otimes\R$ associ\'ee \`a la norme
$\|\cdot\|$ (c'est-\`a-dire la mesure de Lebesgue si l'on identifie
$\Omega\otimes\R$ \`a $\R^n$ {\em via} une base orthonorm\'ee pour
$\|\cdot\|$). Cette d\'efinition s'\'etend \`a un corps de nombres et
nous nous concentrons sur la pente maximale
$\widehat{\mu}_{\mathrm{max}}(\overline{E})$ de $\overline{E}$, c'est-\`a-dire le
maximum des pentes des sous-fibr\'es non nuls de $\overline{E}$.

Tandis qu'une comparaison entre le minimum et le covolume d'un
r\'eseau euclidien fait l'objet du premier th\'eor\`eme de Minkowski
(appel\'e aussi lemme de Siegel dans certains contextes), une
comparaison analogue pour le minimum absolu r\'esulte d'un
th\'eor\`eme profond de Zhang, qui joue un r\^ole crucial dans ce
texte. Nous citons ici une forme l\'eg\`erement simplifi\'ee (voir la
partie~\ref{Pzh} pour une discussion plus compl\`ete).

\begin{theo}\label{Tzhi} Pour tout fibr\'e ad\'elique hermitien
$\overline{E}$ de dimension $n\ge1$ sur un corps de nombres nous
avons $$1\le\Lambda(\overline{E},\qb)e^{\widehat{\mu}_{\mathrm{max}}
(\overline{E})}\le\sqrt n.$$\end{theo}

Ici la minoration r\'esulte facilement des d\'efinitions, le contenu
du th\'eor\`eme r\'eside dans la majoration qui m\'erite le nom de
th\'eor\`eme de Minkowski absolu ou de lemme de Siegel absolu (suivant
la terminologie de Roy-Thunder~\cite{rt}).

Nous montrons dans un premier temps que l'ordre de grandeur de cette
majoration ne peut pas \^etre am\'elior\'e. Un r\'esultat analogue sur
$\Q$ portant classiquement le nom de th\'eor\`eme de Minkowski-Hlawka,
notre premier r\'esultat correspond donc \`a un th\'eor\`eme de
Minkowski-Hlawka absolu.

\begin{theo}\label{thlawka} Pour tout $n\ge1$ il existe un fibr\'e ad\'elique
hermitien $\overline{E}$ de dimension $n$ 
sur un corps de nombres avec
$$\Lambda(\overline{E},\qb)e^{\widehat{\mu}_{\mathrm{max}}(\overline{E})}\ge
\sqrt{\frac{n}{e}}.$$\end{theo}

Il s'agit ici encore d'une forme simplifi\'ee (voir
th\'eor\`eme~\ref{mha}) et nous donnerons un proc\'ed\'e explicite
pour trouver un tel $\overline{E}$.

Tournons-nous \`a pr\'esent vers le comportement vis-\`a-vis du
produit tensoriel. En utilisant le th\'eor\`eme~\ref{Tzhi} (de
mani\`ere indirecte) nous d\'emontrons l'in\'egalit\'e centrale
suivante.

\begin{theo}\label{prin} Pour deux fibr\'es ad\'eliques hermitiens
$\overline{E}$ 
et $\overline{F}$ sur un corps de nombres, on a 
$$\Lambda(\overline{E},\qb)e^{-\widehat{\mu}_{\mathrm{max}}(\overline{F})}\le
\Lambda(\overline{E}\otimes\overline{F},\qb).$$\end{theo}

Cet \'enonc\'e s'\'etend imm\'ediatement par r\'ecurrence \`a $N$
fibr\'es puis en le combinant (de mani\`ere directe cette fois) avec
le th\'eor\`eme~\ref{Tzhi} nous aboutissons aux estimations suivantes.

\begin{coro}\label{cor} Soient $N\ge1$ un entier naturel et
$\overline{E}_1,\ldots,\overline{E}_N$ des
fibr\'es ad\'eliques hermitiens sur un corps de nombres. Nous
avons\begin{equation}
\label{lamb}1\le\Lambda\left(\bigotimes_{i=1}^N\overline{E}_i,\qb\right)^{-1}
\prod_{i=1}^N\Lambda(\overline{E}_i,\qb)
\le\left(\prod_{i=2}^N\dim E_i\right)^{1/2}\ \text{et}\end{equation}
\begin{equation}\label{mum}0\le\widehat{\mu}_{\mathrm{max}}
\left(\bigotimes_{i=1}^N\overline{E}_{i}\right)-\sum_{i=1}^N
\widehat{\mu}_{\mathrm{max}}(\overline{E}_{i})\le\frac{1}{2}\sum_{i=1}^N
\log\dim E_i.\end{equation}\end{coro}

L'in\'egalit\'e~\eqref{mum} a d\'ej\`a fait l'objet de plusieurs travaux. Dans
le seul publi\'e \`a ce jour, Chen~\cite{hchen} obtient comme majorant
$\sum_{i=1}^N\log\dim E_i$. Avec le coefficient $1/2$, notre
majoration \'etait connue dans le cas $N=2$, \'etablie
ind\'ependamment par Bost\footnote{\emph{Stability of Hermitian vector
bundles over arithmetic curves and geometric invariant theory}.
Expos\'e au \emph{Chern Institute}, Nankai. Avril 2007.} et
Andr\'e~\cite{Andre}. En outre, dans un travail en cours, Bost et
Chen~\cite{bc} donnent une version plus forte de~\eqref{mum} o\`u le
majorant devient $(1/2)\sum_{i=2}^N\log\dim E_i$ et prouvent donc
l'exact analogue de \eqref{lamb}.

Dans le corollaire, les minorations suivent facilement des
d\'efinitions mais l'on peut se demander si elles peuvent \^etre
strictes. Dans le cas~\eqref{mum} une conjecture de Bost pr\'edit qu'il y a
toujours \'egalit\'e :
$\widehat{\mu}_{\mathrm{max}}(\overline{E}\otimes\overline{F})=
\widehat{\mu}_{\mathrm{max}}(\overline{E})+\widehat{\mu}_{\mathrm{max}}
(\overline{F})$ pour tous fibr\'es ad\'eliques hermitiens $\overline{E}$ et
$\overline{F}$. Ceci reste un probl\`eme 
ouvert : voir \cite{bc} pour des r\'esultats partiels ; plus bas nous
donnons seulement un exemple de situation o\`u il y a \'egalit\'e (en
pr\'esence d'une action de groupe).

En revanche dans le cadre de~\eqref{lamb} nous montrons que la conjecture
analogue se trouve \^etre {\em fausse}.

\begin{theo}\label{nml} Pour tous entiers naturels $n,m\ge2$ il existe
deux fibr\'es ad\'eliques hermitiens sur un corps de nombres
$\overline{E}$ et $\overline F$ de dimensions respectives $n$ et $m$ avec 
$$\Lambda(\overline{E}\otimes\overline{F},\qb)<\Lambda(\overline{E},\qb)
\Lambda(\overline{F},\qb).$$\end{theo}

Dans le langage de l'article d'Andr\'e \cite{Andre}, ceci signifie que
la propri\'et\'e d'\^etre num\'eriquement effectif ne se conserve pas
par produit tensoriel.

Dans la fin du texte, nous utilisons le th\'eor\`eme~\ref{prin} pour
estimer minimum et pente maximale des puissances sym\'etriques ou
ext\'erieures d'un fibr\'e ad\'elique hermitien.
Par exemple nous am\'eliorons la borne
$$\widehat{\mu}_{\mathrm{max}}\left(\overline{S^{\ell}(E)}\right)
\le\ell\left(\widehat{\mu}_{\mathrm{max}}(\overline{E})
+2n\log n\right)$$ due \`a Bost o\`u $n=\dim E$
(voir~\cite[th\'eor\`eme 7.1]{rendiconti}) en rempla\c cant
$2n\log n$ par $2\log n$. Nous montrons aussi par un exemple que
ceci est le bon ordre de grandeur et ne peut pas \^etre remplac\'e par
$c\log n$ si $c<1/2$.
\vskip3mm
\par\textbf{Remerciements}. Nous remercions Jean-Beno{\^{\i}}t Bost et
Huayi Chen ainsi que Yves Andr{\'e} de nous avoir communiqu{\'e} leurs
textes~\cite{bc} et~\cite{Andre} avant publication.

\section{Rappels de th\'eorie des pentes}\label{sectheopentes}
Soit $k$ un corps de nombres. Notons $V(k)$ l'ensemble des places de $k$ et
$k_{\mathbf{A}}$ les ad\`eles de $k$. Pour $v\in V(k)$, le corps
$\mathbf{C}_v$ d\'esigne la compl\'etion d'une cl{\^o}ture
alg\'ebrique du compl\'et\'e $k_v$ de $k$. Nous le munissons de la
valeur absolue $|\cdot|_v$ qui \'etend la valeur absolue usuelle de $\Q_v$,
adh\'erence de $\Q$ dans $k_v$. La \emph{formule du produit} s'\'ecrit 
$$\forall\,
x\in\,k\setminus\{0\},\quad\prod_{v\in V(k)}{|x|_v^{[k_v:\Q_v]}}=1.$$

\subsection{Fibr\'es ad\'eliques hermitiens} Soit
$n\in\mathbf{N}$. Soit $(k^n,|\cdot|_{2})$ le couple --- dit
\emph{fibr\'e hermitien standard} --- constitu\'e de l'espace
vectoriel $k^n$ et d'une collection de normes
$|\cdot|_{2}=(|\cdot|_{2,v}\colon \mathbf{C}_v^n\mapsto\mathbf{R}_+)_{v\in
V(k)}$ d\'efinies de la mani\`ere suivante~: pour tout vecteur
$(x_1,\ldots,x_n)$ de $\mathbf{C}_v^n$ on a \begin{equation*}
|(x_1,\ldots,x_n)|_{2,v}:=\begin{cases}\left(\sum_{i=1}^n{|
x_{i}|_v^{2}}\right)^{1/2} & \text{si $v$ est
archim\'edienne}\\ \max{\{| x_1|_v,\ldots,|
x_n|_v\}} & \text{sinon}.
\end{cases}\end{equation*}
\begin{defi}\label{definitionfva}
Un \emph{fibr\'e ad\'elique hermitien}
$\overline{E}=(E,(\|\cdot\|_{\overline{E},v})_v)$ sur $k$, de
dimension $n$, est la donn\'ee d'un $k$-espace vectoriel $E$ de
dimension $n$ et d'une famille de normes
$\|\cdot\|_{\overline{E},v}$ sur $E\otimes_{k}\mathbf{C}_v$,
index\'ees par les places $v$ de $k$, qui satisfont {\`a} la
condition suivante~: il existe une $k$-base $(e_1,\ldots,e_n)$ de
$E$ et une matrice ad\'elique $a=(a_v)_{v\in
V(k)}\in\GL_n(k_{\mathbf{A}})$ telles que, pour toute place $v\in
V(k)$, la norme sur $E\otimes_{k}\mathbf{C}_v$ est donn\'ee
par \begin{equation*}\label{definormun}\forall\,x=(x_1,\ldots,x_n)
\in\mathbf{C}_v^n,\quad\left\|\sum_{i=1}^n{x_{i}e_{i}}\right
\|_{\overline{E},v}=| a_v(x)|_{2,v}.\end{equation*}
\end{defi}
Cette d\'efinition est \'equivalente \`a celle de
\emph{fibr\'e vectoriel hermitien sur $\spec\mathcal{O}_{k}$}. Si
$k=\mathbf{Q}$, cette notion correspond {\`a} celle de
\emph{r\'eseau euclidien} au sens de la g\'eom\'etrie des
nombres classique. Nous renvoyons le lecteur
{\`a}~\cite{Bost,rendiconti} pour les
propri\'et\'es de ces objets mentionn\'ees ci-apr\`es (rappelons
qu'une base $e_1,\ldots,e_n$ de $E\otimes\mathbf{C}_v$ est dite
orthonorm\'ee si $\|\sum_{i=1}^nx_ie_i\|_{\overline{E},v}=|x|_{2,v}$
pour tout $x=(x_1,\ldots,x_n)\in\mathbf{C}_v^n$).

\subsection{Hauteurs}\label{defh} Soit $\overline{E}$ un fibr\'e
ad\'elique hermitien sur un corps de nombres $k$. 
Soient $K$ une extension alg\'ebrique de $k$ et $x$ un \'el\'ement de
$E\otimes_kK$. Soit $K'\subset K$ une extension finie de $k$ telle que
$x\in E\otimes_{k}K'$. Le $K'$-espace vectoriel $E\otimes_{k}K'$ a une 
structure de fibr{\'e} ad{\'e}lique hermitien induite par celle 
de $\overline{E}$ (pour $w\in V(K')$ au-dessus de $v\in V(k)$ on
identifie $E\otimes K'\otimes{\mathbf C}_w$ \`a $E\otimes{\mathbf
C}_v$). Nous noterons $\Vert.\Vert_{\overline{E},w}$ plut\^ot que 
$\Vert.\Vert_{\overline{E\otimes_{k}K'},w}$ la norme de
$E\otimes_{k}K'$ relative {\`a} $w\in V(K')$.
Avec ces conventions, la \emph{hauteur normalis\'ee} de
$x$ est le nombre r{\'e}el~: $$H_{\overline{E}}(x)=\prod_{v\in V(K')}{\|
x\|_{\overline{E},v}^{[K'_v:\mathbf{Q}_v]/[K':\mathbf{Q}]}}.$$ Cette
d\'efinition ne d\'epend pas de l'extension finie $K'$ choisie. Le
(premier) \emph{minimum} de $E$ sur $K$, not\'e
$\Lambda(\overline{E},K)$, est le nombre
r\'eel~: $$\Lambda(\overline{E},K):=\min{\left\{H_{\overline{E}}(x)\,;\
x\in E\otimes_{k}K\setminus\{0\}\right\}}.$$ On remarque que 
$\Lambda(\overline{E},K)=\inf_{K'|
k}\Lambda(\overline{E},K')$ o{\`u} $K'$ parcourt les
extensions finies de $k$ incluses dans $K$. Dans la suite nous voyons
toujours $k$ comme un sous-corps de $\qb$ et nous int\'eressons
principalement \`a $\Lambda(\overline E,\qb)$ (not\'e
$\exp\left(-\widehat\deg{}_\mathrm{n}^{(1)}(\overline E)\right)$ dans
\cite{bc}). Par ailleurs, si 
$\varphi\colon E\to F$ est une application lin\'eaire entre deux espaces
munis de structures ad\'eliques hermitiennes alors, en chaque place
$v$ de $k$, l'on dispose de la norme d'op\'erateur
$\|\varphi\|_v$ de l'application $(E\otimes\mathbf{C}_v,
\|\cdot\|_{\overline{E},v})\to(F\otimes\mathbf{C}_v,
\|\cdot\|_{\overline{F},v})$ induite par $\varphi$. La d\'efinition de
cette norme implique imm\'ediatement que si $\varphi$ est
\emph{injective} alors
\begin{equation}\label{ineqminima}\Lambda(\overline{F},K)\le
H(\varphi)\Lambda(\overline{E},K)\quad\text{o{\`u}}\quad
H(\varphi):=\prod_{v\in
V(k)}{\|\varphi\|_v^{[k_v:\mathbf{Q}_v]/[k:\mathbf{Q}]}}.\end{equation}
\subsection{Pentes} \'Etant donn\'e un fibr\'e hermitien $\overline{E}$
sur $k$ donn\'e par une matrice ad\'elique $a=(a_{v})_{v\in V(k)}$, 
la \emph{pente (normalis\'ee)} $\widehat{\mu}(\overline{E})$ 
de $\overline{E}$ est le nombre r\'eel $$\widehat{\mu}(\overline{E}):=
-\frac{1}{\dim E}\sum_{v\in V(k)}{\frac{[k_{v}:\mathbf{Q}_{v}]}
{[k:\mathbf{Q}]}\log\vert\det a_{v}\vert_{v}}.$$
La \emph{pente maximale} $\widehat{\mu}_{\mathrm{max}}(\overline{E})$ de
$\overline{E}$ est le maximum des pentes $\widehat{\mu}(\overline{F})$
lorsque $F$ parcourt les sous-espaces vectoriels non nuls de $E$ (le
fibr\'e $\overline{F}$ est $F$ muni des m\'etriques induites par
celles de $\overline{E}$). Le fibr\'e $\overline{E}$ est dit
\emph{semi-stable} si
$\widehat{\mu}_{\mathrm{max}}(\overline{E})=\widehat{\mu}(\overline{E})$.
La pente maximale est atteinte en un unique sous-espace de dimension
maximale et elle est invariante par extension de corps (nous rappelons
l'argument au paragraphe~\ref{act}).

\subsection{Dual et somme directe}\label{dsd}
Soit $\overline{E}$ un fibr\'e ad\'elique hermitien sur $k$. Le
\emph{dual} de $\overline{E}$, not\'e $\overline{E}^{\mathsf{v}}$,
est le fibr\'e ad\'elique hermitien d'espace sous-jacent
$E^{\mathsf{v}}=\Hom_{k}(E,k)$ muni des normes (d'op\'erateurs)
duales de celles sur $E$. On a
$\widehat{\mu}(\overline{E}^{\mathsf{v}})=-\widehat{\mu}(\overline{E})$
(voir~\cite[proposition~$4.21$]{rendiconti}). Soit $\overline{F}$ un
autre fibr\'e ad\'elique hermitien sur $k$. La \emph{somme directe
(hermitienne)} $\overline{E}\oplus\overline{F}$ de $\overline{E}$ et
$\overline{F}$ est le fibr\'e ad\'elique hermitien d'espace
sous-jacent $E\oplus F$ de
normes $$\|(x,y)\|_{\overline{E}\oplus\overline{F},v}=\begin{cases}(\|
x\|_{\overline{E},v}^{2}+\| y\|_{\overline{F},v}^{2})^{1/2}
&\text{si $v\mid\infty$},\\ \max{\{\|
x\|_{\overline{E},v},\| y\|_{\overline{F},v}\}} &
\text{si $v\nmid\infty$}\end{cases}$$ pour tous $x\in
E\otimes_{k}\mathbf{C}_v$ et $y\in F\otimes_{k}\mathbf{C}_v$.
On a $\widehat{\mu}(\overline{E}\oplus\overline{F})=
((\dim E)\widehat{\mu}(\overline{E})+(\dim
F)\widehat{\mu}(\overline{F}))/(\dim E\oplus F)$.

\subsection{Produit tensoriel}\label{produitt} Soient $\overline{E}$ et
$\overline{F}$ deux fibr\'es ad\'eliques hermitiens sur $k$.
Le \emph{produit tensoriel} $\overline{E}\otimes\overline{F}$
est le fibr\'e ad\'elique hermitien d'espace sous-jacent
$E\otimes_{k}F$ et dont la norme sur
$E\otimes F\otimes\mathbf{C}_v$ ($v\in V(k)$) est telle que des
bases orthonorm\'ees $(e_1,\ldots,e_n)$ de
$(E\otimes_{k}\mathbf{C}_v,\|\cdot\|_{\overline{E},v})$ et
$(f_1,\ldots,f_{m})$ de $(F\otimes_{k}\mathbf{C}_v,\|\cdot
\|_{\overline{F},v})$ donnent une base orthonorm\'ee $\{e_{i}\otimes
f_{j}\,;\ 1\le i\le n,\ 1\le j\le m\}$ de $(E\otimes F\otimes
\mathbf{C}_v,\|\cdot\|_{\overline{E}\otimes\overline{F},v})$. On a
la formule $\widehat{\mu}(\overline{E}\otimes\overline{F})=
\widehat{\mu}(\overline{E})+\widehat{\mu}(\overline{F})$
(voir~\cite[proposition~$5.2$]{rendiconti}).

\subsection{Puissance sym\'etrique}\label{puissancesym}
Soient $\overline{E}$ un fibr\'e ad\'elique hermitien sur $k$ et $\ell\ge
1$ un entier. La puissance sym\'etrique $\ell^{\text{\`eme}}$ de
$E$, not\'ee $S^{\ell}(E)$, est un quotient de la puissance
tensorielle $E^{\otimes\ell}$. Cet espace vectoriel est de dimension
$\binom{\ell+n-1}{n-1}$ o{\`u} $n=\dim E$. Il poss\`ede une structure ad\'elique
hermitienne $\overline{S^{\ell}(E)}$ d\'efinie de la mani\`ere
suivante~: soient $v$ une place de $k$ et $(e_1,\ldots,e_n)$ une
base orthonorm\'ee de $E_v=E\otimes\mathbf{C}_v$; la norme
$\|\cdot\|_{\overline{S^{\ell}(E)},v}$ est l'unique norme pour
laquelle la base $e^{i}:=e_1^{i_1}\cdots e_n^{i_n}$, pour
$i=(i_1,\ldots,i_n)\in\mathbf{N}^n$ de somme $|
i|=\ell$, est orthonorm\'ee si $v$ est ultram\'etrique et
orthogonale avec $$\|
e^{i}\|_{\overline{S^{\ell}(E)},v}=\sqrt{i_1!\cdots
i_n!/\ell!}$$ si $v$ est archim\'edienne. Il revient au
m{\^{e}}me de dire que $S^{\ell}(E)$ est muni de la structure quotient
de celle de $\overline{E}^{\otimes\ell}$
(voir~\cite[p. 45--46]{rendiconti}).

\subsection{Puissance ext\'erieure}
\label{secpuissanceexterieure}
Soient $\overline{E}$ un fibr\'e ad\'elique hermitien sur $k$ de
dimension $n\ge 1$ et $\ell\in\{1,\ldots,n\}$. 
La puissance
ext\'erieure $\wedge^{\ell}E$ peut {\^{e}}tre munie de normes
hermitiennes aux places $v$ de $k$ en d\'ecr\'etant qu'une base
orthonorm\'ee $(e_1,\ldots,e_n)$ de $E\otimes_{k}\mathbf{C}_v$
fournit une base orthonorm\'ee $\{e_{i_1}\wedge\cdots\wedge
e_{i_{\ell}}\,;\ 1\le i_1<\cdots<i_{\ell}\le n\}$ de
$\wedge^{\ell}E\otimes_{k}\mathbf{C}_v$. Ceci conf\`ere une
structure de fibr\'e hermitien $\overline{\wedge^{\ell}E}$ {\`a} la
puissance ext\'erieure, qui, contrairement
au cas de la puissance sym{\'e}trique, n'est pas exactement la
structure quotient que l'on obtient par la projection canonique
$E^{\otimes\ell}\to\wedge^{\ell}E$ (un facteur $\sqrt{\ell!}$
diff\'erencie les normes aux places archim\'ediennes). La pente de ce
fibr\'e v\'erifie l'\'egalit\'e
\begin{equation*}\label{pentepe}\widehat{\mu}
\left(\overline{\wedge^{\ell}E}\right)=\ell\widehat{\mu}
(\overline{E})\end{equation*}(une d\'emonstration de cette
formule repose sur l'\'egalit\'e $\det\wedge^{\ell} u=(\det
u)^{\binom{n-1}{\ell -1}}$ o{\`u} $u$ est un endomorphisme d'espace
vectoriel). Lorsque $\ell=n$, cette formule peut {\^{e}}tre utilis{\'e}e  
comme d{\'e}finition de la pente de $\overline{E}$ en termes de
hauteurs. En effet, dans ce cas, $\wedge^{n}E$ est une droite et l'on a
$\widehat{\mu}(\overline{\wedge^{n}E})=-\log
\Lambda(\overline{\wedge^{n}E},k)=
-\log H_{\overline{\wedge^{n}E}}(x)$, pour un \'el\'ement $x$
quelconque de $(\wedge^{n}E)\setminus\{0\}$.

\subsection{In\'egalit\'e de convexit\'e pour les hauteurs}\label{ich}
\begin{lemma}\label{convexe} Soient $N$ un entier $\ge 1$ et
$\overline{E}_1,\ldots,\overline{E}_N$ des fibr\'es
ad\'eliques hermitiens sur $k$. Pour tout $(x_1,\ldots,x_N)\in
E_1\times\cdots\times E_N$, on a \begin{equation*}\left(\sum_{i=1}^N
{H_{\overline{E}_{i}}(x_{i})^{2}}\right)^{1/2}\le H_{\overline{E}_1
\oplus\cdots\oplus\overline{E}_N}(x_1,\ldots,x_N).\end{equation*}
\end{lemma}
\begin{proof}
Par convexit\'e de l'application $x\mapsto\log(1+e^x)$, on a, pour
tout ensemble $J$ de cardinal $D\ge 1$ et toute famille de nombres
r\'eels positifs $(\theta_j)_{j\in J}$, $$1+\left(\prod_{j\in
J}{\theta_{j}}\right)^{1/D}\le\prod_{j\in
J}{(1+\theta_{j})}^{1/D}.$$ Par r\'ecurrence sur $N$, on en
d\'eduit que, pour tout $(a_{i,j})_{1\le i\le N,\,j\in J}\in
(\mathbf{R}_+)^{ND}$, on a $$\sum_{i=1}^N{\left(\prod_{j\in
J}{a_{i,j}}\right)^{1/D}}\le\prod_{j\in
J}{\left(a_{1,j}+\cdots+a_{N,j}\right)}^{1/D}.$$ On applique alors
cette in\'egalit\'e \`a l'ensemble $J$ des plongements
complexes de $k$ et $a_{i,j}=\|x_i\|_{\overline{E}_i,j}^2$ puis l'on
multiplie par les contributions ultram\'etriques pour conclure.\end{proof}

Avec les notations du lemme, nous d\'eduisons en particulier que pour
toute extension alg\'ebrique $K$ de $k$ on a $$\Lambda(\overline{E}_{1}
\oplus\cdots\oplus\overline{E}_{N},K)=\min_{1\le i\le N}
\Lambda(\overline{E}_{i},K).$$

\section{Autour du th\'eor\`eme de Zhang}\label{Pzh}
\subsection{\'Enonc\'e g\'en\'eral}
Soient $\overline{E}$ un fibr\'e ad\'elique hermitien de dimension $n\ge1$ sur un
corps de nombres 
et $i$ un entier avec $1\le i\le n$. On d\'efinit le $i$-\`eme minimum
de Zhang $Z_i(\overline{E})$ comme la borne inf\'erieure des nombres r\'eels
$r$ tels que l'adh\'erence de Zariski de $\{x\in
E\otimes\overline{\mathbf{Q}}\,;\ H_{\overline{E}}(x)\le r\}$ soit de
dimension au moins $i$. Nous avons $\Lambda(\overline{E},\qb)=Z_1(\overline{E})\le
Z_2(\overline{E})\le\cdots\le Z_n(\overline{E})$ et, si l'on note $H_n:=1+1/2+\cdots+1/n$
le nombre harmonique, le th\'eor\`eme des minima successifs de Zhang
dans le cas des fibr\'es ad\'eliques hermitiens s'\'enonce comme suit.

\begin{theo}\label{Tzh} Pour tout fibr\'e ad\'elique hermitien
$\overline{E}$ de dimension $n\ge1$, on a
$$(Z_1(\overline{E})\cdots Z_n(\overline{E}))^{1/n}\le
e^{(H_n-1)/2-\widehat{\mu}(\overline{E})}\le Z_n(\overline{E}).$$\end{theo}

Ce r\'esultat d\'ecoule du th\'eor\`eme~5.2 de \cite{Zhang} avec
$X=\mathbf{P}(E^{\mathsf v})$ et il faut \'evaluer la hauteur de $X$ comme par
exemple dans \cite{zh}. Mentionnons qu'en toute g\'en\'eralit\'e Zhang
traite le cas d'une vari\'et\'e arithm\'etique $X$ avec des
m\'etriques ad\'eliques non n\'ecessairement hermitiennes. Dans ce
cadre \'elargi, les in\'egalit\'es du th\'eor\`eme sont optimales car
il est possible d'avoir $(Z_1(\overline{E})\cdots Z_n(\overline{E}))^{1/n}
=Z_n(\overline{E})$ comme on le voit en prenant $E=k^n$ muni des
m\'etriques du supremum en toutes les places (on a
$Z_i(\overline{E})=1$ pour tout $i$ car les points \`a coordonn\'ees
dans les racines de l'unit\'e forment un ensemble dense de points de
hauteur minimale 1).

Lorsque l'on se limite au cadre hermitien comme ici, l'on peut tout de
m\^eme voir que l'ordre de grandeur de l'encadrement du th\'eor\`eme
ne peut pas \^etre am\'elior\'e. En effet si $\overline{E}$ est le fibr\'e
standard $(k^n,|\cdot|_2)$ alors on a $Z_i(\overline{E})=\sqrt i$ pour tout
$i$. Pour le voir on utilise le lemme~\ref{convexe} de convexit\'e
des hauteurs pour montrer que $H_{\overline{E}}(x)\ge\sqrt i$ d\`es que $x$ a $i$
coordonn\'ees non nulles, avec \'egalit\'e lorsque toutes ces
coordonn\'ees sont des racines de l'unit\'e. Dans ce cas, l'assertion
du th\'eor\`eme s'\'ecrit $$n!^{1/(2n)}\le\exp((H_n-1)/2)\le\sqrt n$$
o\`u $n!^{1/(2n)}\ge\sqrt{n/e}$.

\subsection{Lemme de Siegel absolu}

Le th\'eor\`eme~\ref{Tzh} implique le lemme suivant qui lui-m\^eme
entra\^\i ne le th\'eor\`eme~\ref{Tzhi}.

\begin{lemma} Pour tout fibr\'e ad\'elique hermitien de dimension
$n\ge1$, on a 
$$1\le\Lambda(\overline{E},\qb)e^{\widehat{\mu}_{\mathrm{max}}
(\overline{E})}\le e^{(H_n-1)/2}.$$\end{lemma}

\begin{proof} L'in\'egalit\'e 
$\Lambda(\overline{E},\qb)\le(Z_1(\overline{E})\cdots
Z_n(\overline{E}))^{1/n}$ fournit 
$\Lambda(\overline{E},\qb)e^{\widehat{\mu}(\overline{E})}\le\exp((H_n-1)/2)$
et nous appliquons cette estimation \`a un sous-espace $F$ tel que
$\widehat{\mu}_{\mathrm{max}}(\overline{E})=\widehat{\mu}(\overline F)$ :
$$\Lambda(\overline{E},\qb)
e^{\widehat{\mu}_{\mathrm{max}}(\overline{E})}\le
\Lambda(\overline{F},\overline{\mathbf{Q}})
e^{\widehat{\mu}(\overline{F})}\le\exp((H_{\dim F}-1)/2)\le\exp((H_n-1)/2).$$
La premi\`ere minoration $-\log\Lambda(\overline{E},\qb)
\le\widehat{\mu}_\mathrm{max}(\overline{E})$ traduit simplement
l'in\'egalit\'e tautologique $$\sup\{\widehat{\mu}(\overline{F})\mid 
F\subset E\otimes\qb,\ \dim F=1\}\le\sup\{\widehat{\mu}(\overline{F})\mid
F\subset E\otimes\qb,\ \dim F\ge1\}$$ lorsque l'on sait que la pente
maximale est invariante par extension des scalaires.\end{proof}

Au lieu du th\'eor\`eme~\ref{Tzhi} nous aurions pu donner un
\'enonc\'e interm\'ediaire avec des minima successifs plus classiques
c'est-\`a-dire d\'efinis comme les $Z_i$ en rempla{\c{c}}ant l'adh\'erence
de Zariski par l'espace vectoriel engendr\'e. On obtient alors un
lemme de Siegel absolu exactement de la forme de celui de Roy-Thunder
(avec une meilleure constante). Ceci ne pr\'esente pas d'int\'er\^et
pour ce qui suit car l'on sait que la constante optimale est la m\^eme
dans les deux formulations (voir \cite{rt} et \cite{vaa}).

\subsection{Th\'eor\`eme de Minkowski-Hlawka absolu} \'Etant
donn\'e un fibr\'e ad\'elique hermitien $\overline{E}$ sur $k$, notons
$q(\overline{E})$ le nombre r\'eel 
$\Lambda(\overline{E},\overline{\mathbf{Q}})e^{\widehat{\mu}(\overline{E})}$
(analogue de l'invariant d'Hermite d'un r\'eseau euclidien). Le
r{\'e}sultat que  
nous pr{\'e}sentons dans ce paragraphe est une version plus pr{\'e}cise
du th{\'e}or{\`e}me~\ref{thlawka}. 
\begin{theo}\label{mha}
Pour tout entier $n\ge 1$, pour tout $\varepsilon>0$, il
existe un fibr\'e ad\'elique hermitien $\overline{E}$ de
dimension $n$ (sur un corps de nombres qui d\'epend de $n$) tel
que $\sqrt{n!}^{1/n}(1-\varepsilon)\le q(\overline{E})$. En
particulier il existe $\overline{E}$ de dimension $n$ tel que
$\sqrt{n/e}\le q(\overline{E})$.
\end{theo}

Le th{\'e}or{\`e}me de Minkowski-Hlawka classique 
affirme que, pour tout
$n\ge 2$, il existe un r\'eseau euclidien $\overline E$ sur $\mathbf Q$ de
dimension $n$ avec $\Lambda(\overline{E},\mathbf{Q})
e^{\widehat{\mu}(\overline{E})}\ge(2\zeta(n)/v_n)^{1/n}$ o\`u $v_n$
est le volume de la boule unit\'e de $(\mathbf{R}^n,|\cdot|_2)$. Plus
g\'en\'eralement Thunder \cite{th} a d\'emontr\'e une minoration
analogue avec un fibr\'e $\overline E$ sur un corps de nombres $k$ fix\'e :
$\Lambda(\overline{E},k)e^{\widehat{\mu}(\overline{E})}
\ge c_k\sqrt n+o(1)$ pour un r\'eel $c_k>0$. Toutefois, contrairement
\`a ces r\'esultats o\`u l'existence de
l'exemple s'obtient par un argument de moyenne, nous donnons ici un
exemple \emph{explicite} de $\overline{E}$ satisfaisant la minoration.

\begin{proof}Si $n=1$ alors $q(\overline{E})=1$ pour tout
$\overline{E}$ et le th\'eor\`eme est d\'emontr\'e. Supposons
maintenant $n\ge 2$. Soit $M\in\mathrm{M}_n(\mathbf{Z})$ une
matrice dont aucun mineur (d'ordre quelconque) n'est nul. Par
exemple la matrice de coefficient $(i,j)$ \'egal {\`a} $(i+j)!!$
convient. Choisissons ensuite un
nombre premier $p$ tel que la propri\'et\'e des mineurs de $M$
reste vraie modulo $p$ (par exemple si $p$ est strictement plus
grand que les valeurs absolues des mineurs de $M$). Par ailleurs,
sans restriction, on peut supposer que
$\varepsilon<1-\sqrt{1-1/n}$. Prenons alors $a_1=1$ et
$a_{2},\ldots,a_n\in p^{\mathbf{Q}}$ tels que, pour tout
$i\in\{2,\ldots,n\}$, on a
$\sqrt{i}(1-\varepsilon)<a_{i}\le\sqrt{i}$. La condition sur
$\varepsilon$ entra{\^{\i}}ne $a_1<\cdots<a_n$. Soit
$k=\mathbf{Q}(a_1,\ldots,a_n)$ et, pour $t\in\{1,\ldots,n\}$,
posons $\Delta_{t}$ la matrice diagonale
$\mathrm{diag}(a_1,\ldots,a_{t})$ de
$\mathrm{M}_{t}(k)$. Consid\'erons le $k$-fibr\'e ad\'elique
hermitien $\overline{E}$ \'egal {\`a} $(k^n,|\cdot|_{2})$
sauf en les places $v$ de $k$ au-dessus de $p$ pour lesquelles on
pose $$\forall\,x\in E\otimes_{k}\mathbf{C}_v,\quad\|
x\|_{\overline{E},v}:=\left|\Delta_nM(x)\right|_{2,v}.$$
La pente de $\overline{E}$ vaut $-(1/n)\log\prod_{v\mid
p}{|\det(\Delta_nM)|_v^{[k_v:\mathbf{Q}_v]/[k:\mathbf{Q}]}}=\log
(a_1\cdots a_n)^{1/n}$ car $| a_{i}|_v=a_{i}^{-1}$
pour tout $v\mid p$. Montrons maintenant que
$\Lambda(\overline{E},\overline{\mathbf{Q}})=1$. Soient
$t\in\{1,\ldots,n\}$ et
$x\in\overline{\mathbf{Q}}^n\setminus\{0\}$ ayant exactement $t$
coordonn\'ees non nulles, disons
$x_{j_1},\ldots,x_{j_{t}}$. Soient
$x'=(x_{j_1},\ldots,x_{j_{t}})$ et $m$ la matrice extraite
$t\times t$ de $M$ de lignes $\{1,\ldots,t\}$ et de colonnes
$\{j_1,\ldots,j_{t}\}$. Le vecteur
$\Delta_{t}m(x')$ est le vecteur des $t$ premi\`eres
coordonn\'ees de $\Delta_nM(x)$ et donc, pour
toute place $v$ de $k$ au-dessus de $p$, on a $\|
x\|_{\overline{E},v}\ge|\Delta_{t}m(x')|_{2,v}$. Cette
derni\`ere quantit\'e vaut $\max_{i}{| a_{i}y_{i}|_v}$
o{\`u} les $y_{i}$ sont les coordonn\'ees de
$m(x')$. On obtient alors $$\|
x\|_{\overline{E},v}\ge| a_{t}|_v|
m(x')|_{2,v}=a_{t}^{-1}|x'|_{2,v}=a_{t}^{-1}| x|_{2,v}$$
car $m$ est une isom\'etrie en la place $v$ (choix de $M$). On en d\'eduit
que $$H_{\overline{E}}(x)\ge
a_{t}^{-1}H_{(k^n,|\cdot|_{2})}(x)\ge
a_{t}^{-1}\sqrt{t}\ge 1$$
en utilisant le lemme~\ref{convexe}. Comme
de plus $H_{\overline{E}}(1,0,\ldots 0)=1$, ceci montre que
$\Lambda(\overline{E},\overline{\mathbf{Q}})=1$ puis
$q(\overline{E})=(a_1\cdots a_n)^{1/n}$. La minoration
$q(\overline{E})\ge(1-\varepsilon)\sqrt{n!}^{1/n}$ d\'ecoule alors
du choix des $a_{i}$ et l'in\'egalit\'e $n!>(n/e)^n$
entra{\^{\i}}ne $q(\overline{E})\ge\sqrt{n/e}$ (avec un choix
convenable de $\varepsilon$).
\end{proof}

\section{Minima et produit tensoriel}
L'{\'e}tude du comportement du premier minimum
d'un fibr{\'e} hermitien par produit tensoriel remonte aux
travaux de Kitaoka dans les ann\'ees
soixante-dix~\cite[chapitre~$7$]{Kitaoka}.  
Il a montr{\'e} que $\Lambda(\overline{E}\otimes\overline{F},\mathbf{Q})=
\Lambda(\overline{E},\mathbf{Q})\Lambda(\overline{F},\mathbf{Q})$
 lorsque $\overline{E},\overline{F}$ sont des fibr{\'e}s ad{\'e}liques 
 hermitiens sur $\mathbf{Q}$ avec $\dim E\le 43$.
\textit{A contrario} Steinberg a \'etabli
l'existence, pour tout entier $n\ge 292$, d'un fibr{\'e}
$\overline{E}$ sur $\mathbf{Q}$, de dimension $n$, tel que 
$\Lambda(\overline{E}^{\otimes 2},\mathbf{Q})<
\Lambda(\overline{E},\mathbf{Q})^{2}$ (voir~\cite[p.~$47$]{Milnor}).
Par la suite, Coulangeon a obtenu des
r\'esultats analogues pour des minima de fibr{\'e}s sur
un corps imaginaire quadratique~\cite{Coulangeon}. Dans
cette partie nous examinons le comportement du minimum sur
$\overline{\mathbf{Q}}$ vis-\`a-vis du produit tensoriel. La situation
est plus tranch\'ee : nous montrons que le couple $\{43,292\}$ sur
$\mathbf{Q}$ peut \^etre remplac\'e par $\{1,2\}$.

\subsection{R\'eseau de racines de type $\mathbf{A}_n$}\label{an}

Nous donnons un bref exemple pour lequel on a multiplicativit\'e par
rapport au produit tensoriel. Ce serait d\'ej\`a le cas pour le
fibr\'e standard mais nous proposons une variante qui nous servira
\'egalement \`a illustrer d'autres r\'esultats.

Soient $n$ un entier $\ge 1$ et $A_n$
l'hyperplan $\{(x_1,\ldots,x_{n+1})\,;\
x_1+\cdots+x_{n+1}=0\}$ de $\mathbf{Q}^{n+1}$.
\begin{prop} Le fibr\'e ad\'elique hermitien
$\mathbf{A}_n:=(A_n,|\cdot|_{2})$ poss\`ede la
propri\'et\'e suivante. Pour toute extension alg\'ebrique
$K/\mathbf{Q}$, pour tout 
fibr\'e ad\'elique hermitien $\overline{E}$ sur $\mathbf{Q}$, on
a $$\Lambda(\mathbf{A}_n\otimes
\overline{E},K)=\Lambda(\mathbf{A}_n,K)\Lambda(\overline{E},K)\quad
\text{et}\quad\Lambda(\mathbf{A}_n,K)=\sqrt{2}.$$
\end{prop}
\begin{proof} On peut supposer que $K$ est un
corps de nombres. Pour $x_1,\ldots,x_n\in E_{K}:=E\otimes K$,
l'application $$x:=\sum_{i=1}^n{(e_{i}-e_{n+1})\otimes
x_{i}}\mapsto\iota(x):=(x_1,\ldots,x_n,x_1+\cdots+x_n)$$
est une injection (isom\'etrique) de
$\mathbf{A}_n\otimes\overline{E}_{K}$ dans
$\bigoplus_{i=1}^{n+1}{\overline{E}_{K}}$ et
donc $$H_{\mathbf{A}_n\otimes\overline{E}}(x)=H_{\bigoplus_{i=1}^{n+1}
{\overline{E}}}(\iota(x))\ge\left(H_{\overline{E}}(x_1+\cdots+x_n)^{2}
+\sum_{i=1}^n{H_{\overline{E}}(x_{i})^{2}}\right)^{1/2}$$ par
le lemme~\ref{convexe}. Dans le minorant au moins deux termes ne
sont pas nuls si $x\ne 0$ et, dans ce cas,
$H_{\mathbf{A}_n\otimes\overline{E}}(x)\ge\sqrt{2}\Lambda(\overline{E},K)$
puis $\Lambda(\mathbf{A}_n\otimes\overline{E},K)\ge\sqrt{2}\Lambda
(\overline{E},K)$. L'\'egalit\'e s'obtient en consid\'erant le
vecteur $(e_1-e_{n+1})\otimes x_1$ avec
$H_{\overline{E}}(x_1)=\Lambda(\overline{E},K)$.\end{proof}
De cette d\'emonstration l'on d\'eduit \'egalement que
les vecteurs $x$ de $A_n\otimes E_{K}$
tels que $H_{\mathbf{A}_n\otimes\overline{E}}(x)=
\Lambda(\mathbf{A}_n\otimes\overline{E},K)$
sont \emph{scind\'es}, de la forme $(e_{i}-e_{j})\otimes e$ avec
$i\ne j$ et $H_{\overline{E}}(e)=\Lambda(\overline{E},K)$. 

\subsection{Non-multiplicativit\'e}
Dans ce paragraphe nous d\'emontrons que le
premier minimum sur $\overline{\mathbf{Q}}$ n'est en g{\'e}n{\'e}ral
pas multiplicatif relativement au produit
tensoriel (th\'eor\`eme~\ref{nml}).
\par Soit $q\in\mathbf{Q}\setminus\{0\}$. Pour $\varepsilon\in\{-1,1\}$, 
soit $(2+\varepsilon i)^q$ une
puissance $q^{\text{\`eme}}$ de $2+\varepsilon i$. 
Soit $k$ une extension finie de
$\mathbf{Q}(i)$ contenant $(2+i)^q$ et $(2-i)^q$. Les nombres 
$2+i$ et $2-i$ engendrent les deux
id\'eaux premiers de $\mathbf{Z}[i]$ au-dessus de $5$ et la 
valeur absolue $v$-adique de $2+\varepsilon i$ vaut $1/5$ ou $1$
selon que la place $v$ de $k$ divise ou non la place $2+\varepsilon
i$. Consid\'erons alors le 
fibr\'e ad\'elique hermitien $\overline{E}_q$ sur $k$ d'espace
sous-jacent $k^{2}$, de normes $|\cdot|_{2,v}$ en toutes les
places $v$ de $k$ qui ne divisent pas $5$ et, pour
$\varepsilon\in\{-1,1\}$,  \begin{equation*}\label{norme}
\|(x,y)\|_{\overline{E}_q,v}:=\max{\left(|
x|_v,5^{q}\left\vert x+\varepsilon
y\right|_v\right)}\quad\text{si $v\mid
2+\varepsilon i$}.\end{equation*}En d'autres termes,
$\overline{E}_{q}$ est le fibr{\'e} hermitien donn{\'e} par la matrice 
$$A_{\varepsilon}=\begin{pmatrix} 1 & 0\\
(2+\varepsilon i)^{-q} & \varepsilon (2+\varepsilon
i)^{-q}\end{pmatrix}$$ en une place $v\mid 2+\varepsilon i$ (et 
la matrice identit{\'e} en les autres places). En particulier, on a 
$\widehat{\mu}(\overline{E}_{q})=-(q/2)\log 5$.

\begin{theo}\label{ce} Il existe un nombre r\'eel $\eta>0$ tel que,
pour tout nombre rationnel $q$ v\'erifiant
$\sqrt{2}<5^q<\sqrt{2}+\eta$, on a
$\Lambda(\overline{E}_q,\overline{\mathbf{Q}})=5^q$. En
particulier, pour un tel $q$, on a
$\Lambda(\overline{E}_q^{\otimes2},\overline{\mathbf{Q}})
<\Lambda(\overline{E}_q,\overline{\mathbf{Q}})^{2}$.
\end{theo}

La d\'emonstration utilise le lemme suivant.

\begin{lemma} Soient $q\in\mathbf{Q}$ et $(x,y)\in\overline{\mathbf{Q}}^{2}$
avec $q>0$ et $xy\ne0$. Alors $H_{\overline{E}_q}(x,y)>\sqrt2$.\end{lemma}

\begin{proof} Si $v\mid2+\varepsilon i$ nous avons
$\max(|x|_v,|y|_v)\le\max(|x|_v,|x+\varepsilon y|_v)\le\|(x,y)\|_{
\overline{E}_q,v}$ donc 
$H_{\overline{E}_q}(x,y)\ge H_{(k^{2},|\cdot|_{2})}(x,y)\ge\sqrt2$
o\`u la seconde in\'egalit\'e r\'esulte du
lemme~\ref{convexe}. Supposons qu'il y ait \'egalit\'e
$H_{\overline{E}_q}(x,y)=\sqrt{2}$. Alors
$y=\xi x$ avec $\xi$ racine de l'unit\'e v\'erifiant $|
1+\varepsilon\xi|_v\le 5^{-q}$ pour toute place $v\mid
2+\varepsilon i$ et tout $\varepsilon\in\{-1,1\}$. Ainsi
$(1+\varepsilon\xi)/(2+\varepsilon i)^q$, dont toutes les valeurs
absolues ultram\'etriques sont plus petites que $1$, est un entier
alg\'ebrique. L'\'egalit\'e $1-\xi=-\xi(\overline{1-\xi})$
montre que $(1-\xi)/(2+i)^q$ est aussi un entier
alg\'ebrique. Il en est alors de m{\^{e}}me pour
$2/(2+i)^q=(1-\xi+1+\xi)/(2+i)^q$ ce qui est absurde puisque sa valeur
absolue en une place au-dessus de $2+i$ est $5^q>1$.\end{proof}

\begin{proof}[D\'emonstration du th\'eor\`eme~\ref{ce}] Nous
appliquons le th\'eor\`eme~\ref{Tzh} {\`a} $\overline{E}_{1/5}$ : on a
$Z_{2}(\overline{E}_{1/5})\ge e^{1/4}5^{1/10}>5^{1/4}>\sqrt2$. Par
d\'efinition de $Z_2$, les couples $(x,y)\in\overline{\mathbf{Q}}^{2}$
tels que $H_{\overline{E}_{1/5}}(x,y)\le 5^{1/4}$ appartiennent \`a un
nombre fini de droites. D'apr\`es le lemme, il existe donc $\eta>0$ tel
que, pour tout $(x,y)\in\overline{\mathbf{Q}}^2$ avec $xy\ne 0$, on a
$H_{\overline{E}_{1/5}}(x,y)\ge\sqrt{2}+\eta$. Choisissons
maintenant $q\in\mathbf{Q}$ tel que $\sqrt{2}<5^q<\sqrt{2}+\eta$,
ce qui implique $q>1/5$. Pour tous nombres alg\'ebriques $x,y$ et
toute place $v$, on a
$\|(x,y)\|_{\overline{E}_q,v}\ge\|(x,y)\|_{\overline{E}_{1/5},v}$
et donc $H_{\overline{E}_q}(x,y)\ge
H_{\overline{E}_{1/5}}(x,y)$. Si $xy\ne 0$ ces hauteurs sont
sup\'erieures {\`a} $\sqrt{2}+\eta>5^q$. Si $xy=0$ et
$(x,y)\ne(0,0)$ alors $H_{\overline{E}_q}(x,y)=5^q$. Ceci montre
que $\Lambda(\overline{E}_q,\overline{\mathbf{Q}})=5^q$. Quant
au produit tensoriel $\overline{E}_q^{\otimes 2}$, il s'identifie
{\`a} $k^{4}$ muni de sa structure hermitienne usuelle sauf en les
places $v$ au-dessus de $2+\varepsilon i$, $\varepsilon\in\{-1,1\}$,
pour lesquelles la norme est donn\'ee par le produit de Kronecker
$A_{\varepsilon}\otimes A_{\varepsilon}$~:
$$\|(x,y,z,t)\|_{\overline{E}_q\otimes\overline{E}_q,v}=\max{(|
x|_v, 5^q| x+\varepsilon y|_v,5^q|
x+\varepsilon z|_v, 5^{2q}| x+\varepsilon y+\varepsilon
z+t|_v)}.$$De la sorte on a $\Lambda(\overline{E}_q^{\otimes
2},\overline{\mathbf{Q}})\le
H_{\overline{E}_q\otimes\overline{E}_q}(1,0,0,-1)=\sqrt{2}\cdot
5^q$ et la seconde assertion du th\'eor\`eme s'ensuit.
\end{proof}

Le th\'eor\`eme~\ref{nml} s'obtient en choisissant pour $\overline{E}$
et $\overline{F}$ des sommes directes hermitiennes d'un
$\overline{E}_q$ fix\'e et d'un fibr\'e quelconque de
minimum absolu $\ge5^{q}$. En calculant le minimum d'une somme directe
comme indiqu\'e apr\`es le lemme~\ref{convexe}, nous trouvons 
$\Lambda(\overline{E},\overline{\mathbf{Q}})=
\Lambda(\overline{F},\overline{\mathbf{Q}})=5^q$ et
$\Lambda(\overline{E}\otimes\overline{F},\overline{\mathbf{Q}})\le
\Lambda(\overline{E}_q^{\otimes2},\overline{\mathbf{Q}})<5^{2q}$.

\section{Pentes et produit tensoriel}
\subsection{Actions de groupes}\label{act}
Nous donnons dans ce paragraphe un calcul de pente maximale en
pr\'esence d'une action de groupe. Nous disons qu'un \emph{groupe $G$ agit
sur un fibr\'e vectoriel ad\'elique $\overline E$} lorsque $G$ agit sur
l'espace vectoriel $E$ et que, pour tout $g\in G$ et toute place $v$ de $k$,
l'automorphisme de $E\otimes_k{\mathbf C}_v$ donn\'e par $g$ est une
isom\'etrie pour la norme $\|\cdot\|_{{\overline E},v}$. De plus, l'action est
dite :\begin{itemize}\item \emph{irr\'eductible} si les seuls sous-espaces
$G$-stables de $E$ sont $\{0\}$ et $E$ ;\item \emph{g\'eom\'etriquement
irr\'eductible} si l'action de $G$ sur $E\otimes_k\overline k$ est
irr\'eductible.\end{itemize} Avec cette terminologie, nous avons
l'\'enonc\'e suivant.

\begin{prop}\label{actgr} Soient ${\overline E}$ et ${\overline F}$
deux fibr\'es vectoriels ad\'eliques hermitiens. Soit $G$ un groupe
agissant sur ${\overline E}$ de fa\c con g\'eom\'etriquement
irr\'eductible. Alors $\overline E$ est semi-stable et
$\widehat\mu_{\mathrm{max}}({\overline E}\otimes{\overline F})
=\widehat\mu_{\mathrm{max}}({\overline E})+\widehat\mu_{\mathrm{max}}({\overline
F})$.\end{prop}

Nous utilisons dans la d\'emonstration le lemme ci-dessous de
th\'eorie des repr\'esentations.

\begin{lemma} Soient $E$ et $F$ deux espaces vectoriels de dimension
finie sur un corps $k$. Soient $G$ un groupe agissant sur $E$ de
mani\`ere g\'eom\'etriquement irr\'eductible et $W$ un sous-espace
$G$-stable de $E\otimes F$. Alors il existe un sous-espace $F_0$ de
$F$ tel que $W=E\otimes F_0$.\end{lemma}

\begin{proof} Nous raisonnons par l'absurde et consid\'erons un espace
$F$ de dimension minimale pour lequel le lemme est faux (\`a $E$
fix\'e). Nous choisissons ensuite un sous-espace $W$ de dimension
minimale qui viole l'\'enonc\'e. Ceci entra\^\i ne d'une part que le
seul sous-espace $F_1$ de $F$ tel que $W\subset E\otimes F_1$ est $F$
lui-m\^eme et d'autre part que $W$ est irr\'eductible pour l'action de
$G$ : s'il contenait un sous-espace strict non nul et $G$-stable
celui-ci serait de la forme $E\otimes F_2$ et $W/(E\otimes F_2)\subset
E\otimes(F/F_2)$ fournirait un contre-exemple de dimension
inf\'erieure. Par cons\'equent si $\ell\in F^{\mathsf v}$ alors
$\mathrm{id}_E\otimes\ell$ induit un $G$-morphisme $W\to E$ qui ne peut
\^etre que nul ou bijectif. S'il est nul $W\subset E\otimes\Ker\ell$ donne
$\Ker\ell=F$ donc $\ell=0$. Par suite, pour tous
$\ell,\ell'\in F^{\mathsf v}\setminus\{0\}$, nous obtenons un $G$-automorphisme
$\varphi=(\mathrm{id}_E\otimes\ell')_{|W}\circ(\mathrm{id}_E\otimes
\ell)_{|W}^{-1}$ de $E$. Par le lemme de Schur, $\varphi$ s'\'ecrit
$\lambda\mathrm{id}_E$ pour $\lambda\in k$ (c'est ici qu'intervient
l'hypoth\`ese d'irr\'eductibilit\'e sur $\overline k$ : il existe
$\lambda\in\overline k$ tel que $\varphi_{\overline
k}-\lambda\mathrm{id}_{E\otimes\overline k}$ est non inversible donc nul
donc $\varphi_{\overline k}=\lambda\mathrm{id}_{E\otimes\overline k}$ puis
$\lambda\in k$). Nous en d\'eduisons
$(\mathrm{id}_E\otimes\ell')_{|W}=\lambda(\mathrm{id}_E\otimes\ell)_{|W}$ puis
$W\subset E\otimes\Ker(\ell'-\lambda\ell)$ et, comme ci-dessus, cela
entra\^\i ne $\ell'-\lambda\ell=0$. On conclut $\dim F=1$ mais comme $\dim
W=\dim E$ l'inclusion $W\subset E\otimes F$ ne peut pas \^etre stricte
et c'est la contradiction cherch\'ee.\end{proof}

L'autre ingr\'edient de notre d\'emonstration est l'existence, pour un
fibr\'e ad\'elique hermitien ${\overline E}$, d'un unique sous-espace $V$ de
pente et dimension maximales (voir \cite[lemme 5.12]{rendiconti})
parfois appel\'e sous-espace \emph{d\'estabilisant}. En particulier, si un
groupe $G$ agit sur ${\overline E}$, ceci montre imm\'ediatement que $V$ est
$G$-stable (si $g\in G$ le sous-espace $g(V)$ a m\^eme pente et m\^eme
dimension que $V$). On en d\'eduit donc (voir \cite[proposition
5.17]{rendiconti}) que si l'action est irr\'eductible alors
${\overline E}$ est semi-stable (car $V=E$). Ceci permet \'egalement
de voir que la pente maximale est invariante par extension de corps
{\em via} l'action d'un groupe de Galois.

\begin{proof}[D\'emonstration de la proposition~\ref{actgr}] Soit $W$ l'unique
sous-espace de $E\otimes F$ avec $\widehat\mu(\overline W)=\widehat\mu_{\mathrm
{max}}({\overline E}\otimes{\overline F})$ et de
dimension maximale parmi les sous-espaces de m\^eme pente. Par
unicit\'e $W$ est $G$-stable donc par le lemme il est de la forme
$E\otimes F_0$. Par suite nous
avons $$\widehat\mu_{\mathrm{max}}({\overline E}\otimes{\overline
F})=\widehat\mu(\overline W)=\widehat\mu({\overline E})+\widehat\mu({\overline
{F_0}})\le\widehat\mu_{\mathrm{max}}({\overline E})+\widehat\mu_{\mathrm
{max}}({\overline F})$$ et le r\'esultat en d\'ecoule car
l'in\'egalit\'e contraire est toujours vraie.\end{proof}

Comme application directe, nous voyons que $\mathbf{A}_n$ est
semi-stable et que, pour tout
fibr\'e ad\'elique hermitien $\overline{E}$ sur un corps de nombres, on a
$$\widehat{\mu}_{\mathrm{max}}(\mathbf{A}_n\otimes\overline{E})=
\widehat{\mu}_{\mathrm{max}}(\mathbf{A}_n)+\widehat{\mu}_{\mathrm{max}}
(\overline{E}).$$
En effet le groupe sym\'etrique
$\mathfrak{S}_{n+1}$ sur $n+1$ \'el\'ements agit sur $\mathbf{A}_n$ de
mani\`ere g\'eo\-m\'e\-tri\-que\-ment irr\'eductible (en permutant les
coordonn\'ees).

\subsection{} Le but de ce paragraphe est de d\'emontrer le
th\'eor\`eme~\ref{prin} et ses cons\'equences directes.

\begin{proof}[D{\'e}monstration du th{\'e}or{\`e}me~\ref{prin}] Soit
$K$ une extension finie de $k$. On choisit 
$x=\sum_{i=1}^{\ell}{e_i\otimes f_i}\in E\otimes F\otimes K$ tel
que i)
$H_{\overline{E}\otimes\overline{F}}(x)=\Lambda(\overline{E}\otimes\overline{F},K)$
et ii) l'entier $\ell$ est minimal pour cette propri\'et\'e. La
minimalit\'e de $\ell$ implique que les espaces vectoriels
$E_1:=\vect_{K}(e_1,\ldots,e_{\ell})$ et
$F_1:=\vect_{K}(f_1,\ldots,f_{\ell})$ sont de dimension
$\ell$. Montrons que\footnote{L'argument qui suit est classique. On
le trouve en substance au chapitre~$7$ du livre de
Kitaoka~\cite{Kitaoka}.} 
\begin{equation}\label{ineqf}H_{\overline{E}\otimes\overline{F}}(x)\ge
\sqrt{\ell}\exp{\left\{-\left(\widehat{\mu}(\overline{E_1})
+\widehat{\mu}(\overline{F_1})\right)\right\}}.\end{equation}
On le fait place par place. On choisit des bases orthonorm\'ees de
$\overline{E_1}$ et $\overline{F_1}$ en une place $v$ et l'on
note $X_v$ (\emph{resp}. $Y_v$) la matrice de $(e_1,\ldots,e_{\ell})$
(\emph{resp}. $(f_1,\ldots,f_{\ell})$) dans les bases
orthonorm\'ees choisies. Alors $\| x
\|_{\overline{E}\otimes\overline{F},v}$ vaut $|
{}^{\mathrm{t}}X_vY_v|_{2,v}$, \emph{i.e.} la racine carr\'ee de
la somme des carr\'es des valeurs absolues des coefficients de
${}^{\mathrm{t}}X_vY_v$ (norme de Hilbert-Schmidt) si $v$ est
archim\'edienne et le maximum des valeurs absolues des
coefficients de ${}^{\mathrm{t}}X_vY_v$ si $v$ est
ultram\'etrique. Or, comme cons\'equence de l'in\'egalit\'e
arithm\'etico-g\'eom\'etrique dans le cas archim\'edien et
de la simple in\'egalit\'e ultram\'etrique dans le cas
ultram\'etrique, on
a \begin{equation*}\left|\det X_v\right|_v^{1/\ell}
\left|\det Y_v\right|_v^{1/\ell}
=\left|\det{}^{\mathrm{t}}X_vY_v\right|_v^{1/\ell}\le\|
x\|_{\overline{E}\otimes\overline{F},v}\times\begin{cases}\ell^{-1/2}
& \text{si $v\mid\infty$},\\ 1 & \text{si
$v\nmid\infty$}.\end{cases}\end{equation*} On obtient alors~\eqref{ineqf}
par produit car la matrice ad\'elique $(X_v)_{v\in V(K)}$ d\'efinit la
structure hermitienne de $\overline{E_1}$ ({\em idem} pour les $Y_v$ et
$\overline{F_1}$). Ceci \'etant \'etabli, on applique le
th\'eor\`eme~\ref{Tzhi} \`a $\overline{E_1}$ :
$\Lambda(\overline{E_1},\overline{\mathbf{Q}})\le\sqrt{\ell}
\exp{\left\{-\widehat{\mu}(\overline{E_1})\right\}}$. En
majorant $\widehat{\mu}(\overline{F_1})$ par
$\widehat{\mu}_{\mathrm{max}}(\overline{F})$ et
$\Lambda(\overline{E},\overline{\mathbf{Q}})$ par
$\Lambda(\overline{E_1},\overline{\mathbf{Q}})$, on
obtient $$\Lambda(\overline{E},\overline{\mathbf{Q}})
e^{-\widehat{\mu}_{\mathrm{max}}(\overline{F})}
\le\Lambda(\overline{E}\otimes\overline{F},K)$$ pour
toute extension $K/k$. Ceci permet de conclure.
\end{proof}

\begin{rema} Nous venons de d\'emontrer le th\'eor\`eme~\ref{prin} \`a
l'aide du th\'eor\`eme~\ref{Tzhi}. Inversement, ce dernier est contenu
dans le  th\'eor\`eme~\ref{prin} comme le montre l'argument suivant
(inspir\'e de celui du  th\'eor\`eme 4.6 de \cite{bc}) : soit $V$ le
sous-espace d\'estabilisant de $E$ ; on a
$\Lambda(\overline{V}\otimes\overline{V}^{\mathsf{v}},
\overline{\mathbf{Q}})\le\sqrt{\dim
V}\le\sqrt n$ en consid\'erant l'identit\'e ; le
th\'eor\`eme~\ref{prin} donne
$\Lambda(\overline{V}\otimes\overline{V}^{\mathsf{v}},\overline{\mathbf
Q})\ge\Lambda(\overline{V},\overline{\mathbf{Q}})
e^{-\widehat{\mu}_{\mathrm{max}}(\overline{V}^{\mathsf{v}})}\ge\Lambda(\overline{E},
\overline{\mathbf{Q}})e^{\widehat{\mu}_{\mathrm{max}}(\overline{E})}$ 
car 
$\widehat{\mu}_{\mathrm{max}}(\overline{V}^{\mathsf{v}})=\widehat{\mu}
(\overline{V}^{\mathsf{v}})=-\widehat{\mu}(\overline{V})
=-\widehat{\mu}_{\mathrm{max}}(\overline{E})$
par semi-stabilit\'e de $\overline V$ et de son dual. On
obtient ainsi $\Lambda(\overline{E},\overline{\mathbf{Q}})
e^{\widehat{\mu}_{\mathrm{max}}(\overline{E})}\le\sqrt{n}$.
\end{rema}

Nous d\'eduisons de cet \'enonc\'e une version l\'eg\`erement plus
pr\'ecise du corollaire~\ref{cor}.

\begin{coro}\label{theoprecis}Soient $N\ge 1$ un entier naturel et
$\overline{E}_1,\ldots,\overline{E}_{N}$ des fibr\'es
ad\'eliques hermitiens, de dimensions respectives
$n_1,\ldots,n_{N}\ge 1$. Alors on a 
$$\Lambda\left(\overline{E}_{1}\otimes\cdots\otimes
\overline{E}_{N},\qb\right)^{-1}
\prod_{i=1}^N\Lambda(\overline{E}_i,\qb)
\le\exp{\left(\frac{1}{2}\sum_{i=2}^{N}{(H_{n_{i}}-1)}\right)}$$
et
$$\widehat{\mu}_{\mathrm{max}}
(\overline{E}_1\otimes\cdots\otimes\overline{E}_{N})\le\left
(\sum_{i=1}^{N}{\widehat{\mu}_{\mathrm{max}}(\overline{E}_i)}
\right)+\frac1{2}(H_{n_1\cdots n_{N}}-1).$$\end{coro}

Nous pouvons \'egalement citer la cons\'equence suivante du
th\'eor\`eme~\ref{prin}.

\begin{coro}\label{corominima}
Pour tous fibr\'es ad\'eliques hermitiens
$\overline{E}_1,\ldots,\overline{E}_N$, on a
$$\exp{\left\{-\sum_{i=1}^{N}{\widehat{\mu}_{\mathrm{max}}
(\overline{E}_i)}\right\}}\le\Lambda(\overline{E}_1\otimes
\cdots\otimes\overline{E}_N,\overline{\mathbf{Q}}).$$
\end{coro}

\section{Puissances sym\'etriques}
Dans cette partie, nous \'etablissons le r\'esultat suivant.
\begin{theo}\label{pentesym}
Soient $\ell\ge 1$ un entier et $\overline{E}$ un fibr\'e
ad\'elique hermitien sur $k$, de dimension $n\ge 1$. Soit
$p(n,\ell)$ le ppcm des coefficients multinomiaux $\{\ell!/i!\,;\
i\in\mathbf{N}^n\ \text{et}\ | i|=\ell\}$ (voir
l'appendice). On
a \begin{equation*}0\le\widehat{\mu}_{\mathrm{max}}\left(\overline
{S^{\ell}(E)}\right)-\ell\widehat{\mu}_{\mathrm{max}}(\overline{E})\le
\frac1{2}\log\binom{\ell+n-1}{n-1}+\log
p(n,\ell)\end{equation*}et \begin{equation*}1\le\Lambda\left(
\overline{S^{\ell}(E)},\overline{\mathbf{Q}}\right)^{-1}\Lambda(
\overline{E},\overline{\mathbf{Q}})^{\ell}\le p(n,\ell)\exp\left(
\frac{\ell-1}{2}(H_n-1)\right).\end{equation*}\end{theo}
La d\'emonstration du th\'eor\`eme~\ref{pentesym} repose sur le
r\'esultat suivant. 
\begin{lemma}\label{lemmaquetredeux}
Avec les donn\'ees du th\'eor\`eme~\ref{pentesym}, on a
$$\Lambda(\overline{E}^{\otimes\ell},\overline{\mathbf{Q}})\le
p(n,\ell)\Lambda\left(\overline{S^{\ell}(E)},\overline{\mathbf{Q}}\right).$$
\end{lemma}
\begin{proof}
La projection canonique $(E^{\mathsf{v}})^{\otimes\ell}\to
S^{\ell}(E^{\mathsf{v}})$ fournit une injection
isom\'etrique $$\iota\colon \left(S^{\ell}(E^{\mathsf{v}})
\right)^{\mathsf{v}}\hookrightarrow\left((E^{\mathsf{v}})^{\otimes\ell}
\right)^{\mathsf{v}}\simeq E^{\otimes\ell}.$$ Soit
$\theta\colon S^{\ell}(E)\to\left(S^{\ell}(E^{\mathsf{v}})\right)^{\mathsf{v}}$
l'application lin\'eaire d\'efinie de la mani\`ere suivante~:
soient $x_1\cdots x_{\ell}\in S^{\ell}(E)$ ($x_i\in E$) et
$\varphi_1\cdots\varphi_{\ell}\in S^{\ell}(E^{\mathsf{v}})$
($\varphi_i\in E^{\mathsf{v}}$). Posons $$\left(\theta(x_1\cdots
x_{\ell})\right)(\varphi_1\cdots\varphi_{\ell}):=\frac1{\ell
!}\sum_{\sigma\in\mathfrak{S}_{\ell}}\prod_{i=1}^{\ell}
\varphi_i(x_{\sigma(i)})$$ ($\mathfrak{S}_{\ell}$
est le groupe sym\'etrique de $\{1,\ldots,\ell\}$). Cette formule a
un sens car, par sym\'etrie, elle ne d\'epend pas du choix de
l'ordre dans lequel sont les $x_i$. La d\'efinition de $\theta$ se
prolonge naturellement {\`a} des \'el\'ements quelconques de
$S^{\ell}(E)$ et $S^{\ell}(E^{\mathsf{v}})$ par
bilin\'earit\'e. L'application $\theta$ est un isomorphisme
d'espaces vectoriels et l'in\'egalit\'e~\eqref{ineqminima}
appliqu\'ee {\`a} $\iota\circ\theta$
donne $$\Lambda(\overline{E}^{\otimes\ell},\overline{\mathbf{Q}})\le
H(\theta)\Lambda\left(\overline{S^{\ell}(E)},
\overline{\mathbf{Q}}\right).$$ Montrons
que $H(\theta)=p(n,\ell)$. Soit $\theta_v$ le prolongement de
$\theta$ {\`a} $S^{\ell}(E)\otimes_{k}\mathbf{C}_v$. On a une
\'ecriture alternative de $\theta_v$, valable pour toute base
$(e_1,\ldots,e_n)$ de $E\otimes_{k}\mathbf{C}_v$, de base duale
$(\phi_1,\ldots,\phi_n)$, de la forme (les notations sont celles
du \S~\ref{puissancesym})~:
\begin{equation*}\label{defexplicite}\theta_v\left(\sum_{|
i|=\ell}{a_ie^i}\right)\left(\sum_{|
i|=\ell}{b_i\phi^i}\right)=\sum_{|
i|=\ell}{\frac{i!}{\ell
!}a_ib_i}\end{equation*}($a_i,b_i\in k$). Pour utiliser cette formule, 
nous choisissons une base orthonorm\'ee $(e_1,\ldots,e_n)$ de
$E\otimes\mathbf{C}_v$ et la base duale est automatiquement
orthonorm\'ee. L'application $\theta_v$ est une isom\'etrie si $v$ est
archim\'edienne. En effet, en une telle place $v$, en \'ecrivant $(i!/\ell
!)a_ib_i=\sqrt{(i!/\ell !)}a_i\times\sqrt{(i!/\ell !)}b_i$ et
en utilisant l'in\'egalit\'e de Cauchy-Schwarz, on
a$$\|\theta_v(\sum_{|
i|=\ell}{a_ie^i})\|_{\overline{S^{\ell}(E^{\mathsf{v}})^{\mathsf{v}}},v}\le\|\sum_{|
i|=\ell}{a_ie^i}\|_{\overline{S^{\ell}(E)},v}.$$ Le choix
$b_i:=\overline{a_i}$ donne l'\'egalit\'e et $\theta_v$ est
une isom\'etrie en $v$. Si $v$ est ultram\'etrique, 
les bases $(e^i)_{|i|=\ell}$ et $(\phi^i)_{| i|=\ell}$ sont
orthonorm\'ees et par in\'egalit\'e ultram\'etrique, on
a $$\|\theta\|_v=\sup_{x\in S^{\ell}(E\otimes_{k}\mathbf{C}_v)
\setminus\{0\}}{\left(\frac{\|\theta_v(x)\|_{\overline{S^{\ell}
(E^{\mathsf{v}})^{\mathsf{v}}},v}}{\|x\|_{\overline{S^{\ell}(E)},v}}\right)}
=\max_{|i|=\ell}{\left|\frac{i !}{\ell!}\right|_v}$$ (l'\'egalit\'e
est satisfaite comme on le voit en consid\'erant $x=e^i$, $i$ r\'ealisant le
maximum). Il ne reste plus qu'{\`a} observer que le dernier maximum
qui appara\^\i t vaut $| p(n,\ell)|_v^{-1}$ et que le
produit de ces nombres (renormalis\'es avec les degr\'es locaux)
sur les places ultram\'etriques $v$ de $k$ vaut $p(n,\ell)$.
\end{proof}
\begin{proof}[D\'emonstration du th\'eor\`eme~\ref{pentesym}]
La positivit{\'e} de la diff{\'e}rence des pentes maximales est bien 
connue (voir~\cite[proposition~$7.1$]{rendiconti}). 
Pour la majoration de cette diff{\'e}rence,  le corollaire~\ref{corominima} 
minore $\Lambda(\overline{E}^{\otimes\ell},\overline{\mathbf{Q}})$ par
$\exp{\{-\ell\widehat{\mu}_{\mathrm{max}}(\overline{E})\}}$ et le
th\'eor\`eme de Zhang~\ref{Tzhi} donne 
$$\Lambda\left(\overline{S^{\ell}(E)},\overline{\mathbf{Q}}\right)
\le\binom{\ell+n-1}{n-1}^{1/2}\exp{\left\{-\widehat{\mu}_{\mathrm{max}}
\overline{(S^{\ell}(E)})\right\}}.$$
On conclut avec le lemme~\ref{lemmaquetredeux}. En ce qui concerne
le minimum de $\overline{S^{\ell}(E)}$, la majoration par 
$\Lambda(\overline{E},\overline{\mathbf{Q}})^{\ell}$ s'obtient simplement en
consid{\'e}rant $x^{\ell}\in
S^{\ell}(E\otimes\overline{\mathbf{Q}})\setminus\{0\}$ 
pour $x\in(E\otimes\overline{\mathbf{Q}})\setminus\{0\}$. On a 
$$\Lambda(\overline{S^{\ell}(E)},\overline{\mathbf{Q}})\le 
H_{\overline{S^{\ell}(E)}}(x^{\ell})\le 
H_{\overline{E}^{\otimes\ell}}(x\otimes\cdots\otimes x)=
H_{\overline{E}}(x)^{\ell}$$et l'on fait ensuite tendre 
$H_{\overline{E}}(x)$ vers $\Lambda(\overline{E},\overline{\mathbf{Q}})$.
Pour la minoration, on utilise encore le lemme~\ref{lemmaquetredeux}
et l'on minore $\Lambda(\overline{E}^{\otimes\ell},\overline{\mathbf{Q}})$ 
avec le corollaire~\ref{theoprecis}.\end{proof}
Afin de juger de la finesse de cette estimation, voici un calcul de la
pente maximale d'une puissance sym\'etrique lorsque
$\overline{E}=(\mathbf{Q}^n,|\cdot|_2)$. 
\begin{prop}\label{propquatredeux}Soient $n,\ell\ge 1$ deux
entiers. Posons $\lambda:=[\ell/n]$. Pour toute extension
alg{\'e}brique $K$ de $\mathbf{Q}$, on a \begin{equation*}
-\log\Lambda\left(S^{\ell}(\mathbf{Q}^n,|\cdot|_2),K\right)
=\widehat{\mu}_{\mathrm{max}}\left(S^{\ell}(\mathbf{Q}^n,
|\cdot|_2)\right)=\frac1{2}\log\frac{\ell
!}{\lambda!^n\left(\lambda+1\right)^{\ell-n\lambda}}\cdotp\end{equation*}
\end{prop}
\begin{proof}
Soient $e_1,\ldots,e_n$ les vecteurs de la base
canonique de $\mathbf{Q}^n$. Alors
$S^{\ell}(\mathbf{Q}^n,|\cdot|_2)$ est la somme directe
hermitienne des droites $\mathbf{Q}.e_1^{i_1}\cdots
e_n^{i_n}$ pour $i=(i_1,\ldots,i_n)\in\mathbf{N}^n$ de
longueur $| i|=\ell$. Ainsi on
a \begin{equation}\label{eqpentemax}\widehat{\mu}_{\mathrm{max}}
(S^{\ell}(\mathbf{Q}^n,|\cdot|_2))=\max_{|
i|=\ell}{\{\widehat{\mu}(\mathbf{Q}.e_1^{i_1}\cdots
e_n^{i_n})\}}=\max_{|
i|=\ell}{\left\{\frac1{2}\log\frac{\ell!}{i
!}\right\}}\end{equation}(voir~\cite[p.~$66$]{rendiconti}). Par ailleurs
$\Lambda\left(S^{\ell}(\mathbf{Q}^n,|\cdot|_2),K\right)$
est le minimum des hauteurs des $e_1^{i_1}\cdots
e_n^{i_n}$, qui vaut $\min\sqrt{i!/\ell !}$ (voir \S~\ref{ich}).
Ceci \'etablit la premi\`ere
\'egalit\'e de la proposition~\ref{propquatredeux}. Quant {\`a}
la seconde, consid\'erons le multiplet $j=(j_1,\ldots,j_n)$
qui r\'ealise le maximum dans~\eqref{eqpentemax}. Pour tous
entiers $h,m\in\{1,\ldots,n\}$, on a $j_{h}\le j_{m}+1$ comme on le
voit en posant $j'=j-e_{h}+e_{m}$ (lorsque $j_{h}\ne 0$) et en
utilisant $j!\le j'!$. Les valeurs prises par les coordonn\'ees de
$j$ sont donc soit $a$ soit $a+1$ pour un certain entier naturel
$a$. Posons $N=\card\{h\,;\ j_{h}=a+1\}$. On a $N=\ell-na$ car
$| j|=\ell$. Si $N=n$ alors $a+1=\ell/n=\lambda$ et la
proposition est d\'emontr\'ee. Si $N\le n-1$ alors
l'\'egalit\'e $\ell=na+N$ est la division euclidienne de $\ell$
par $n$ et donc $a=\lambda$. Il suffit alors d'observer que
$j!=a!^{n-N}(a+1)!^N=a!^n(a+1)^N$ pour conclure.
\end{proof}

Le th\'eor\`eme~\ref{pentesym} et la
proposition~\ref{propquatredeux} donnent un encadrement de la quantit\'e
$\Delta(n,\ell):=\sup_{\overline{E}}{\left(\widehat{\mu}_{\mathrm{max}}
\left(\overline{S^{\ell}(E)}\right)-\ell\widehat{\mu}_{\mathrm{max}}
(\overline{E})\right)}$
o{\`u} $\overline{E}$ parcourt tous les fibr\'es ad\'eliques
hermitiens de dimension $n$. Plus pr\'ecis\'ement, on
a\begin{equation*}\label{encadrementasymp}\begin{matrix}
\widehat{\mu}_{\mathrm{max}}\left(S^{\ell}(\mathbf{Q}^n,
|\cdot|_2)\right)&\le&\Delta(n,\ell)&\le&
\frac1{2}\log\binom{\ell+n-1}{n-1}+\log p(n,\ell)\\ \parallel & & &
&\parallel \\ \frac{\ell}{2}\log n+\mathrm{o}(\ell) & \le &
\Delta(n,\ell) & \le & \ell H_{n-1}+\mathrm{o}(\ell)
\end{matrix}\end{equation*}
lorsque $\ell\to+\infty$. L'\'egalit\'e de gauche provient de la
formule de Stirling et celle de droite est d\'emontr\'ee au
\S~\ref{subsecconse} o\`u nous verrons aussi $\Delta(n,\ell)\le2\ell \log n$.
\begin{rema}
Le d\'ebut de la d\'emonstration de la
proposition~\ref{propquatredeux} montre \'egalement que la puissance
sym\'etrique $S^{\ell}(\mathbf{Q}^n,|\cdot|_2)$ est
semi-stable si et seulement si $\ell$ ou $n$ \'egale $1$. En effet,
la pente $\widehat{\mu}(S^{\ell}(\mathbf{Q}^n,|\cdot|_2))$ est
la moyenne des pentes des $\mathbf{Q}.e_1^{i_1}\cdots
e_n^{i_n}$ (la pente d'une somme directe hermitienne est une moyenne
pond\'er\'ee des pentes, voir \S~\ref{dsd}) tandis
que~\eqref{eqpentemax} montre que la pente maximale est le maximum de
ces pentes. Par cons\'equent, l'\'egalit\'e
$\widehat{\mu}_{\mathrm{max}}(S^{\ell}(\mathbf{Q}^n,|\cdot|_2))
=\widehat{\mu}(S^{\ell}(\mathbf{Q}^n,|\cdot|_2))$
n'est possible que si $i!$ reste constant sur
$\{i\in\mathbf{N}^n\,;\ | i|=\ell\}$. On voit alors que les
seules possibilit\'es sont $\ell=1$ ou $n=1$ en consid\'erant les
multiplets $(\ell,0,\ldots,0)$ et $(\ell-1,1,0,\ldots,0)$ (lorsque
$n\ge 2$).
\end{rema}

\section{Puissances ext\'erieures}
Soient $\overline{E}$ un fibr\'e ad\'elique hermitien sur $k$ de
dimension $n\ge 1$ et $\ell\in\{1,\ldots,n\}$. 

\begin{theo} On a
\begin{equation*}\ell\widehat{\mu}(\overline{E})
\le\widehat{\mu}_{\mathrm{max}}(\overline{\wedge^{\ell}E})
\le\ell\widehat{\mu}_{\mathrm{max}}(\overline{E})+\frac1{2}
\log\frac{n!}{(n-\ell)!}\end{equation*}et \begin{equation*}
-\log\Lambda(\overline{\wedge^{\ell}E},\overline{\mathbf{Q}})
\le-\ell\log\Lambda(\overline{E},\overline{\mathbf{Q}})+\frac
{\ell-1}{2}(H_n-1)+\frac1{2}\log\ell!.\end{equation*}
\end{theo}
\begin{proof}
La minoration de la pente maximale d\'ecoule de la formule donnant
la pente de $\overline{\wedge^{\ell}E}$ en fonction de celle de
$\overline{E}$ (\S~\ref{secpuissanceexterieure}). Pour la majoration, consid\'erons l'application
$\psi\colon{\wedge^{\ell}E}\to E^{\otimes\ell}$ induite par la forme
multilin\'eaire altern\'ee $$(x_1,\ldots,x_{\ell})\mapsto
\sum_{\sigma\in\mathfrak{S}_{\ell}}{\varepsilon(\sigma)x_{\sigma(1)}
\otimes\cdots\otimes x_{\sigma(\ell)}}$$
($\varepsilon(\sigma)$ est la signature de la
permutation $\sigma$). L'application $\psi$ est injective et ses
normes d'op\'erateur sont $\|\psi\|_v=1$ si $v\nmid\infty$
et $\|\psi\|_v=\sqrt{\ell !}$ si $v\mid\infty$. Par
cons\'equent, on a 
$\Lambda(\overline{E}^{\otimes\ell},\overline{\mathbf{Q}})\le\sqrt{\ell
!}\cdot\Lambda(\overline{\wedge^{\ell}E},\overline{\mathbf{Q}
})$ d'apr\`es~\eqref{ineqminima}. En utilisant le
corollaire~\ref{corominima} et la majoration 
$\Lambda(\overline{\wedge^{\ell}E},\overline{\mathbf{Q}})\le
\sqrt{\binom{n}{\ell}}\exp{\left(-\widehat{\mu}_{\mathrm{max}}
(\overline{\wedge^{\ell}E})\right)}$
du th\'eor\`eme~\ref{Tzhi}, nous obtenons la majoration
souhait\'ee de la pente maximale de $\overline{\wedge^{\ell}E}$. Si
nous appliquons plut\^ot le corollaire~\ref{theoprecis}
nous obtenons la majoration de
$-\log\Lambda(\overline{\wedge^{\ell}E},\overline{\mathbf{Q}})$ du
th\'eor\`eme.
\end{proof}
Dans la minoration de la pente maximale de $\overline{\wedge^{\ell}E}$ 
donn{\'e}e par ce th{\'e}or{\`e}me,  il n'est pas possible en g{\'e}n{\'e}ral
de remplacer $\widehat{\mu}(\overline{E})$ par
$\widehat{\mu}_{\mathrm{max}}(\overline{E})$ 
car si $\ell=n$ alors
$\widehat{\mu}_{\mathrm{max}}(\overline{\wedge^{\ell}E})
=\widehat{\mu}(\overline{\wedge^{\ell}E})=\ell\widehat{\mu}(\overline{E})$.
\par Revenons au r\'eseau de racines
$\mathbf{A}_n$ introduit au \S~\ref{an}.
\begin{prop}
Soient $n\ge 1$ et $\ell\in\{1,\ldots,n\}$ des entiers. Pour tout
fibr\'e ad\'elique hermitien $\overline{E}$ sur $\mathbf Q$, on a
$$\Lambda(\wedge^{\ell}{\mathbf{A}_n\otimes\overline{E}},\overline{
\mathbf{Q}})=\sqrt{\ell+1}\cdot\Lambda(\overline{E},\overline{\mathbf{Q}})\quad\text{et}\quad\widehat{\mu}_{\mathrm{max}}(\wedge^{\ell}\mathbf{A}_{n}\otimes\overline{E})=-\frac{\ell}{2n}\log(n+1)+\widehat{\mu}_{\mathrm{max}}(\overline{E}).$$
\end{prop}
\begin{proof}
En utilisant la base canonique $(e_1,\ldots,e_{n+1})$ de
$\mathbf{Q}^{n+1}$, nous avons une description de la puissance
ext\'erieure $(\wedge^{\ell}A_n)\otimes E$ de la mani\`ere
suivante. Soit $$x=\sum_{1\le i_1<\cdots<i_{\ell}\le
n+1}{e_{i_1}\wedge\cdots\wedge e_{i_{\ell}}\otimes
x_{i_1,\ldots,i_{\ell}}}$$ un vecteur de
$\wedge^{\ell}\mathbf{Q}^{n+1}\otimes
E\otimes\overline{\mathbf{Q}}$. Si $i_1,\ldots,i_{\ell}$ sont des
entiers quelconques de $\{1,\ldots,n\}$, posons
$x_{i_1,\ldots,i_{\ell}}:=0$ si deux indices $i_{j}$ sont \'egaux
et $x_{i_1,\ldots,i_{\ell}}:=\varepsilon(\sigma)x_{\sigma(i_1),
\ldots,\sigma(i_{\ell})}$
si $\sigma$ est la permutation de $\{i_1,\ldots,i_{\ell}\}$, de
signature $\varepsilon(\sigma)$, telle que
$\sigma(i_1)<\cdots<\sigma(i_{\ell})$. Avec cette convention, les
\'equations qui d\'ecrivent $\wedge^{\ell}(A_n)\otimes E$ sont
$(\star)\ \sum_{h=1}^{n+1}{x_{i_1,\ldots,i_{\ell-1},h}}=0$ pour tout
$(i_1,\ldots,i_
{\ell-1})$ tel que $1\le i_1<\cdots<i_{\ell-1}\le n$. Si
$x\in\wedge^{\ell}(A_n)\otimes E\otimes\overline{\mathbf
Q}\setminus\{0\}$, consid\'erons 
$x_{i_1,\ldots,i_{\ell}}$ un coefficient non nul de $x$ avec $1\le
i_1<\cdots<i_{\ell}\le n$. Pour toute partie $J$ de
$\{i_1,\ldots,i_{\ell}\}$ {\`a} $\ell-1$ \'el\'ements, les
\'equations~$(\star)$ montrent qu'il existe
$h_{J}\in\{1,\ldots,n+1\}\setminus\{i_1,\ldots,i_{\ell}\}$ tel que
$x_{J,h_{J}}\ne 0$. Comme il y a $\ell$ possibilit\'es pour $J$ et
que les ensembles $J\cup\{h_{J}\}$ sont tous distincts, il y a au
moins $\ell+1$ vecteurs $x_{i_1,\ldots,i_{\ell}}$ qui ne sont pas
nuls. En une place $v$ d'un corps de d\'efinition de $x$, la norme
de $x$ vaut $\left(\sum{\|
x_{i_1,\ldots,i_{\ell}}\|_{\overline{E},v}^{2}}\right)^{1/2}$
si $v\mid\infty$ et $\max{\|
x_{i_1,\ldots,i_{\ell}}\|_{\overline{E},v}}$ sinon. On obtient
alors la minoration
$H_{\wedge^{\ell}\mathbf{A}_n\otimes\overline{E}}(x)\ge\sqrt{\ell+1}
\cdot\Lambda(\overline{E},\overline{\mathbf{Q}})$ par le lemme de
convexit\'e~\ref{convexe}, ce qui donne
$\Lambda(\wedge^{\ell}{\mathbf{A}_n\otimes\overline{E}},\overline
{\mathbf{Q}})\ge\sqrt{\ell+1}\cdot\Lambda(\overline{E},\overline{\mathbf{Q}})$.
L'\'egalit\'e
s'obtient en choisissant $x$ de la forme
$(e_1-e_{n+1})\wedge\cdots\wedge(e_{\ell}-e_{n+1})\otimes y$ avec
$H_{\overline{E}}(y)$ qui se rapproche de
$\Lambda(\overline{E},\overline{\mathbf{Q}})$. Pour les pentes 
maximales, le groupe sym\'etrique $\mathfrak{S}_{n}$
agit encore sur $\wedge^{\ell}\mathbf{A}_{n}$ de mani{\`e}re 
g{\'e}om{\'e}triquement irr{\'e}ductible (voir~\cite[proposition~$3.12$]{fh}). 
On applique alors la proposition~\ref{actgr} qui donne 
$\widehat{\mu}_{\mathrm{max}}(\wedge^{\ell}\mathbf{A}_{n}\otimes\overline{E})
=\ell\widehat{\mu}_{\mathrm{max}}(\mathbf{A}_{n})+
\widehat{\mu}_{\mathrm{max}}(\overline{E})$. 
La pente maximale de $\mathbf{A}_{n}$ se calcule au
moyen du d{\'e}terminant de ce fibr{\'e} 
$$\widehat{\mu}_{\mathrm{max}}(\mathbf{A}_{n})=
\widehat{\mu}(\mathbf{A}_{n})=
\frac{1}{n}\widehat{\mu}(\wedge^n\mathbf{A}_n)=
-\frac{1}{n}\log\Lambda(\wedge^n\mathbf{A}_n,\overline{\mathbf Q})=
-\frac{1}{2n}\log(n+1)$$ et la proposition est d\'emontr\'ee. 
\end{proof}
\section{Appendice~: ppcm des multinomiaux}
\subsection{}
Soient $n,\ell$ deux entiers strictement
positifs. Soit \begin{equation}\label{ppcmm}p(n,\ell):=\ppcm\left(\frac{\ell
!}{i_1!\cdots i_n!}\,;\ i_{j}\in\mathbf{N}\ \text{et}\
i_1+\cdots+i_n=\ell\right).\end{equation} Dans cette partie,
nous nous proposons d'\'etablir le th\'eor\`eme suivant (le
crochet d\'esigne la partie enti\`ere).

\begin{theo}\label{thmppcm}
On a $$p(n,\ell)=\prod_{j=1}^{n-1}{\frac{\ppcm
\left(1,2,3,\ldots,\left[(\ell+n-1)/j\right]\right)}{\ell+j}}.$$
\end{theo}

Le cas $n=2$ figure d\'ej\`a dans un article de
Williams~\cite{Williams}. Nous n'avons pas trouv\'e de
r\'ef\'erence dans la litt\'erature pour le cas g\'en\'eral.
Aussi donnons-nous ci-dessous une d\'emonstration (\'el\'ementaire)
du th\'eor\`eme~\ref{thmppcm}.

Pour $n$ et $\ell$ comme ci-dessus et $x\ge1$ un nombre r\'eel, nous
utilisons les notations suivantes :
$$q(n,\ell)=\frac{(\ell+n-1)!}{\ell!}p(n,\ell),\qquad
d(x)=\ppcm(1,2,\ldots,[x]),$$ $$r(n,\ell)=\prod_{k=1}^{n-1}
d\left(\frac{\ell+n-1}{
k}\right)=\ppcm(j_1\cdots j_{n-1};\ 1\le hj_h\le
\ell+n-1),$$ $$s(n,\ell)=\ppcm(j_1\cdots j_{n-1};\ 1\le
j_{n-1}<j_{n-2}<\cdots<j_1\le\ell+n-1).$$

Nous allons montrer les divisibilit\'es $r(n,\ell)\mid s(n,\ell)\mid
q(n,\ell)\mid r(n,\ell)$ et donc l'\'egalit\'e de ces entiers
naturels. En particulier $q(n,\ell)=r(n,\ell)$ est le
th\'eor\`eme~\ref{thmppcm}.

\subsubsection*{}$\bullet$ $r(n,\ell)\mid s(n,\ell)$

Soient $j_1,\ldots,j_{n-1}$ avec $1\le hj_h\le\ell+n-1$ pour $1\le h\le
n-1$. D\'efinissons des entiers $j_h'$ par le proc\'ed\'e suivant : si
$j_1',\ldots,j_{h-1}'$ sont construits on choisit un multiple $j_h'$ de
$j_h$ avec $j_h'\not\in\{j_1',\ldots,j_{h-1}'\}$ et $1\le
j_h'\le\ell+n-1$ (ceci est possible car $1\le hj_h\le\ell+n-1$ montre qu'au
moins $h$ multiples de $j_h$ appartiennent \`a $[1,\ell+n-1]$). Alors les
$j_h'$ sont deux \`a deux distincts et $j_1\cdots j_{n-1}\mid
j_1'\cdots j_{n-1}'\mid s(n,\ell)$.

\subsubsection*{}$\bullet$ $s(n,\ell)\mid q(n,\ell)$

Soient $0=j_n<j_{n-1}<j_{n-2}<\cdots<j_1<j_0=\ell+n$. Le nombre $j_h$
divise $j_h\binom{j_{h-1}-1}{j_h}=\frac{(j_{h-1}-1)!}{(j_h-1)!
(j_{h-1}-j_h-1)!}$ donc par produit
$$j_1\ldots j_{n-1}\mid\frac{(j_0-1)!}{(j_0-j_1-1)!\cdots
(j_{n-2}-j_{n-1}-1)!(j_{n-1}-1)!}=\frac{(\ell+n-1)!}{i_1!\cdots
i_n!}\mid q(n,\ell)$$ avec $i_h=j_{h-1}-j_h-1\ge0$ de sorte que
$i_1+\cdots+i_n=j_0-j_n-n=\ell$.

\subsubsection*{}$\bullet$ $q(n,\ell)\mid r(n,\ell)$ : ceci d\'ecoule
par r\'ecurrence sur $n$ du lemme suivant (avec $q(1,\ell)=1$).

\begin{lemma} On a $q(n,\ell-1)\mid q(n,\ell)$ et $q(n,\ell)\mid
d(1+\ell/(n-1))q(n-1,\ell+1)$.\end{lemma}
\begin{proof}
Soit $(i_1,\ldots,i_n)\in\mathbf{N}^n$ tel que
$i_1+\cdots+i_n=\ell -1$. Par d\'efinition de $q(n,\ell)$ et
pour tout $j\in\{1,\ldots,n\}$, on a
$q(n,\ell)\frac{(i_{j}+1)\prod_{h=1}^n{i_{h}!}}{(\ell+n-1)!}\in\mathbf{N}$
ce qui, en sommant sur $j$, conduit
{\`a} $$q(n,\ell)\frac{\prod_{h=1}^n{i_{h}!}}{(\ell+n-2)!}\in\mathbf{N}.$$ Ceci
montre la premi\`ere assertion. Pour la seconde, nous transitons par
des int\'egrales, comme Nair\cite{nair}
dans le cas $n=2$. Consid\'erons le simplexe
$$S_n:=\{(x_1,\ldots,x_n)\in[0,1]^n;\ x_1+\cdots+x_n\le1\}.$$
Si $i_1+\cdots+i_n=\ell$ et
$i_1\ne0$ nous avons $$ I:=i_1\int_{S_n}x_1^{i_1-1}x_2^{i_2}\cdots
x_n^{i_n}\,\mathrm{d}x_1\ldots\mathrm{d}x_n=\frac{i_1!\cdots
i_n!}{(\ell+n-1)!}.$$ Pour le voir, effectuons le changement de
variables $x_j=ty_j$ ($j<n$) et
$x_n=1-t$ qui donne $$I=\left(i_1\int_{S_{n-1}}y_1^{i_1-1}y_2^{i_2}\cdots
y_{n-1}^{i_{n-1}}\,\mathrm{d}y_1\ldots
\mathrm{d}y_{n-1}\right)\int_0^1(1-t)^{i_n}t^{i_1+\cdots+i_{n-1}+n-2}\,\mathrm{d}t.$$
Par int\'egrations par parties successives l'int\'egrale sur $t$ vaut
$i_n!(\ell+n-i_n-2)!/(\ell+n-1)!$ et l'on conclut par
r\'ecurrence. D'autre part, nous avons
aussi \begin{equation*}\begin{split}I&=
\int_{S_{n-1}}(1-x_2-\cdots-x_n)^{i_1}x_2^{i_2}\cdots
x_n^{i_n}\,\mathrm{d}x_2\ldots\mathrm{d}x_n\\&=\sum_{\genfrac{}{}{0pt}{}
{j=(j_2,\ldots,j_n)}{j_2+\cdots+j_n\le i_1}}
\alpha_j\int_{S_{n-1}}x_2^{i_2+j_2}\cdots
x_n^{i_n+j_n}\,\mathrm{d}x_2\ldots\mathrm{d}x_n\end{split}\end{equation*}
avec $\alpha_j\in\mathbf Z$. Fixons une famille $j=(j_2,\ldots,j_n)$ et
notons $i_1',\ldots,i_{n-1}'$ la suite ordonn\'ee
des nombres $i_2+j_2,\ldots,i_n+j_n$. Nous avons
$i_1'+\cdots+i_{n-1}'=i_2+\cdots+i_n+j_2+\cdots+j_n\le
i_1+i_2+\ldots+i_n=\ell$ donc $i_1'\le\ell/(n-1)$. Ainsi $i_1'+1$
divise $d(1+\ell/(n-1))$. Maintenant le
nombre \begin{align*}(i_1'+1)\int_{S_{n-1}}x_2^{i_2+j_2}
\cdots x_n^{i_n+j_n}\,\mathrm{d}x_2\ldots\mathrm{d}x_n&=(i_1'+1)\int_{S_{n-1}}
x_1^{i_1'}\cdots x_{n-1}^{i_{n-1}'}\,\mathrm{d}x_1\ldots\mathrm{d}x_{n-1}\\ &=
\frac{(i_1'+1)!i_2'!\cdots i_{n-1}'!}{(i_1'+\cdots+i_{n-1}'+n-1)!}\end{align*}
appartient \`a $q(n-1,i_1'+\cdots i_{n-1}'+1)^{-1}\mathbf{Z}\subset
q(n-1,\ell+1)^{-1}\mathbf{Z}$. En combinant tout ceci nous avons
$i_1!\cdots i_n!/(\ell+n-1)!\in d(1+\ell/(n-1))^{-1}q(n-1,\ell+1)^{-1}
\mathbf{Z}$ qui donne le lemme.\end{proof}

\subsection{Cons\'equences}\label{subsecconse} Il est bien connu
que le th\'eor\`eme des nombres premiers entra{\^{\i}}ne une
estimation asymptotique de la \emph{fonction de Tchebychev de seconde
esp\`ece} $\psi(x):=\log\ppcm(1,2,\ldots,[x])$~:
$\psi(x)=x+\mathrm{o}(x)$ lorsque $x\to+\infty$. Ce r\'esultat
conduit {\`a} l'estimation
asymptotique \begin{equation*}p(n,\ell)=\exp{\left\{\ell
H_{n-1}+\mathrm{o}(\ell)\right\}},\quad
\ell\to+\infty\end{equation*}o{\`u}, comme plus haut, $H_{n-1}$
d\'esigne le nombre 
harmonique. Il peut {\^{e}}tre
int\'eressant d'avoir un encadrement effectif de $p(n,\ell)$. Par
exemple, en minorant $p(n,\ell)$ par le maximum des coefficients
multinomiaux, on a une minoration simple et asymptotiquement
pr\'ecise~:
$$ n^{\ell}\binom{\ell+n-1}{n-1}^{-1}\le p(n,\ell) $$ (on rappelle que
la somme des multinomiaux vaut $n^{\ell}$). En sens inverse, un
th\'eor\`eme de Rosser \& Schoenfeld~\cite[th{\'e}or{\`e}me~$12$]{rs}
implique $\psi(x)\le x\log(2\sqrt{2})$ pour tout nombre r\'eel $x\ge
0$, qui, {\`a} son tour, fournit la majoration suivante. 

\begin{prop}
Pour tous entiers $n,\ell\ge 1$ on a $\log p(n,\ell)\le \frac{3}{2}\ell\log n$.
\end{prop}
\begin{proof}L'{\'e}nonc{\'e} \'etant trivial si $n$ ou $\ell$ {\'e}gale
$1$, on peut supposer $\ell,n\ge 2$. Si $\ell\le n^{3/2}$ on
utilise la borne $p(n,\ell)\le\ell !$ qui implique
$p(n,\ell)\le\ell^{\ell}\le (n^{3/2})^{\ell}$. Lorsque
$\ell>n^{3/2}$ on utilise le th{\'e}or{\`e}me~\ref{thmppcm} et la
majoration de $\psi$ ci-dessus qui permettent de majorer $\log
p(n,\ell)$ par
$(\ell+n-1)H_{n-1}\log(2\sqrt{2})-\log\left((\ell+1)\cdots(\ell+n-1)\right)$.
Par une \'etude de fonction, on a
$H_{n-1}\log(2\sqrt2)\le(\log(n-1)+1)\log(2\sqrt{2})\le 
\frac{3}{2}\log n$ d{\`e}s que $n\ge 7$. La majoration
$H_{n-1}\log(2\sqrt{2})\le \frac{3}{2}\log n$ 
reste vraie pour $n\le6$ par v\'erification directe. En utilisant
cette borne on trouve
\begin{equation*}(\ell+n-1)H_{n-1}\log(2\sqrt2)\le\frac{3}{2}\ell\log
n+\frac{3}{2}(n-1)\log n\le\frac{3}{2}\ell\log
n+\log\ell^{n-1}\end{equation*}car $n^{3/2}\le\ell$ et l'on
conclut en majorant $\ell^{n-1}$ par le produit des $\ell+i$,
$i\in\{1,\ldots,n-1\}$.\end{proof}

En utilisant $\binom{\ell+n-1}{n-1}\le n^{\ell}$ et
$H_{n}-1\le\log n$, cette borne et le th{\'e}or{\`e}me~\ref{pentesym}
donnent les estimations suivantes. 
\begin{prop}Soient $\overline{E}$ un fibr\'e
ad\'elique hermitien sur $k$, de dimension $n\ge 1$, et $\ell$ un
entier $\ge 1$.
On a \begin{equation*}\widehat{\mu}_{\mathrm{max}}\left(\overline
{S^{\ell}(E)}\right)\le\ell\widehat{\mu}_{\mathrm{max}}(\overline{E})+2\ell\log n
\quad\text{et}\quad\Lambda(
\overline{E},\overline{\mathbf{Q}})^{\ell}\le n^{2\ell}
\Lambda\left(\overline{S^{\ell}(E)},\overline{\mathbf{Q}}\right).\end{equation*}
\end{prop}

\bibliographystyle{plain}

\end{document}